\documentclass{article}

\usepackage{mathtools,amssymb, amsthm,bm,latexsym,tikz}
\usepackage{blindtext}
\usepackage{multirow}
\usepackage{graphicx,proof}
\usepackage{hyperref}

\usepackage[natbibapa]{apacite}
\newtheorem{corollary}{Corollary}
\newtheorem{remark}{Remark}
\def\PA{\mbox{PA}}
\def\PAF{\mbox{\bf PA}}
\def\PATWO{\PA^{\hspace{-1 pt}2}}
\def\ZFC{\mbox{ZFC}}
\def\ZFCTWO{\ZFC^{\hspace{1 pt}2}}
\def\ZF{\mbox{ZF}}
\def\ZFF{\mbox{\bf ZF}}
\def\ZFTWO{\ZF^{\hspace{1 pt}2}}
\def\IA{\mbox{IA}}
\def\IO{\mbox{IO}}

\def\ISO{\mbox{\rm ISO}}
\def\Iso{\mbox{\rm ISO}}
\def\AUT{\mbox{AUT}}
\def\Aut{\mbox{AUT}}

\def\oN{\Bbb N}

\title{Philosophical Uses of Categoricity Arguments}

\author{Penelope Maddy\\ {University of California, Irvine}

\and {Jouko V\"a\"an\"anen}\\ {University of Helsinki}}

\begin{document}

\maketitle

\begin{abstract}
This book addresses the viability of categoricity arguments in philosophy by focusing with some care on the nature of the specific conclusions that a sampling of prominent figures have attempted to draw – the same theorem might successfully support one such conclusion while failing to support another.  It begins with the historical cases of Dedekind, Zermelo, and Kreisel, casting doubt on received readings of the latter two, at least, and highlighting the success of all three in achieving what are argued to be their actual goals. These earlier uses of categoricity arguments are then compared and contrasted with more recent work of Parsons and the co-authors Button and Walsh.  Highlighting the roles of first- and second-order theorems, of external and (two varieties of) internal theorems, the book eventually concludes that categoricity arguments have been more effective in historical cases that reflect philosophically on internal mathematical matters than in more recent uses for questions of pre-theoretic metaphysics.
\end{abstract}





\section{Introduction}\label{sec:int}

Mathematicians and philosophers have appealed to categoricity arguments in a surprisingly varied range of contexts.  One familiar example calls on second-order categoricity in an attempt 
to show that the Continuum Hypothesis, despite its formal independence, has a determinate truth value, but this doesn’t exhaust the uses of categoricity even in set theory, not to mention its appearance in various roles in discussions of arithmetic.  Here we compare and contrast a sampling of these deployments to get a sense of when these arguments tend to succeed and when they tend to fail.  Our story begins with two historical landmarks, Dedekind and Zermelo, on arithmetic and set theory, respectively, and ends with two leading contemporary writers, Charles Parsons and the co-authors Tim Button and Sean Walsh, again on arithmetic and set theory, respectively.  In between, we pause over the well-known contribution of Georg Kreisel.  In each case,\footnote{With the exception of 
Section \ref{sec:kre}, where no new mathematics is involved.} we ask: what does the author set out to accomplish, philosophically?; what do they actually do (or what can be done), mathematically?; and does what’s done (or can be done) accomplish what they set out to do?  We find this focus on context  illuminating:  these authors have qualitatively different philosophical goals, and what works for one might not work for another.

\section{Dedekind in `Was sind und was sollen die Zahlen? (1888)}\label{sec:ded}  

\subsection{What does he set out to accomplish?}  Dedekind’s motivation for his treatment of natural numbers in `Was sind und was sollen die Zahlen' (\cite{02693454}) is explicit: ‘in science nothing capable of proof ought to be believed without proof’ (p. 790).\footnote{Unless otherwise indicated, all references in this section are to the English translation of \cite{02693454} in \cite[pp. 796-833]{Ewald2005-EWAFKT-4}.} Ideally, such a proof should not depend on the vagaries of geometric intuition or temporal intuition or anything of that sort; rather, the ‘number-concept’ and the resulting proofs should emerge as ‘an immediate product of the pure laws of thought’ (pp. 790-791).  This is what Dedekind means when he says that ‘arithmetic [is] merely a part of logic’ (p. 790), so it’s important to note that these ‘pure laws of thought’ include what we would regard as set formation:
\begin{quote}
It very frequently happens that different things, a, b, c, … for some reason  can be considered from a common point of view, can be associated in the mind, and we say that they form a $\textit{system}$.  (\S 2, p. 797)
\end{quote}

\noindent (Of course, a ‘system’ here is what we’d call a ‘set’.)  In sum, then, Dedekind is out to show that the concept of number can be characterized and the theorems of arithmetic proved using only logical notions, in his generous sense of ‘logical’.  This would complete the project, initiated in \cite{rd}, of founding the calculus and higher analysis without appeal to geometric or temporal intuition.\footnote{The weak set theory of \cite{02693454} is enough for the construction of reals via Dedekind cuts in \cite{rd}.}

\subsection{What does he actually do?}\label{subsec:ded2} Dedekind begins by articulating some of his logical laws of thought, in particular, some pre-theoretic metaphysical truths about sets.\footnote{\label{Dede} Fortunately, difficult questions about the status of these sets -- e.g., are they understood realistically or idealistically? -- can be set aside for present purposes.} The most important of these, obviously, is the above-noted assumption that `very frequently' some things `can be considered from a common point of view' and thus collected into a set. (If this isn't Unlimited Comprehension, Dedekind tells us nothing about how it differs; only later does he recognize the danger (see below).) From there, he develops an early theory of sets:  extensionality and the existence of singletons (\S 2), unions (\S 8) and intersections of families (\S 17).  He introduces functions (`mappings', \S 21) as entities distinct from sets, assumes that identity functions exist, that the restriction of a function to a subset is still a function, and that functions can be composed (implicitly).

He then introduces the key idea of a `$\phi$-chain' (\S 37):\footnote{In \cite[pp. 100-101]{dk}, Dedekind explains the importance of this notion to his critic, Keferstein, who thought it could be eliminated.  Dedekind points out that a simply infinite system, $(N, a, \phi)$, could be contained in a larger system, $S$, with some `arbitrary additional elements $t$, to which the mapping $\phi$' could be applied and a subset, $T$, containing $t$ and closed under $\phi$.  The question is how `to cleanse our system $S ...$ of such alien intruders $t$ as disturb every vestige of order and to restrict it to $N$?'  We might try saying `an element $n$ belongs to the sequence $N$ if and only if, starting with the element [a] and counting on and on steadfastly $...$ through a finite number of iterations of the mapping $\phi ...$ I actually reach the element $n ...$ But this $...$ would $...$ contain the most pernicious and obvious kind of vicious circle$...$ Thus, how can I, without presupposing any arithmetic knowledge, give an unambiguous conceptual foundation to the distinction between the elements $n$ and the elements $t$?  Merely through consideration of $\textit{chains} ...$ and$...$ by means of these, completely!'  He notes that Frege's ancestral `agrees in $\textit{essence}$ with my notion of chain $...$ only, one must not be put off by his somewhat inconvenient terminology'. }  

\newtheorem{definition}{Definition}

\begin{definition}

If $\phi$ is a function from S to S, and $K \subseteq S$, then K is a \emph{$\phi$-chain} iff $\forall {x}\in {K}(\phi({x})\in{K})$.

\end{definition}

\noindent This notion in hand, in  Dededkind defines 

\begin{definition}
	
\textit{S} is \emph{simply infinite}	iff there is a one-to-one $\phi$ from \textit{S} to \textit{S} and an $s \in {S}$ such that

\begin{enumerate}
	\item there is no ${x}\in {S}$ with $\phi({x})={s}$,
	
	\item $S=\bigcap\{{X}\mid{s}\in{X} \mbox{ and }{X}\mbox{  is a }\phi\mbox{-chain}\}$.

\end{enumerate}
\end{definition}

\noindent  By this point, he's defined Dedekind-infinite (\S 64) and proved that there's an infinite set in the notorious \S 66:  the collection of `all things which can be objects of my thought' is infinite because a function that sends each element therein to the thought of that element in fact maps the set to a proper subset of itself (at least one element is not a thought,`e.g., my own ego')!  He then shows that every infinite set contains a simply infinite set.  

So, finally, where are the numbers?  Here Dedekind abstracts:

\begin{quote}
If in the consideration of a simply infinite system \textit{N} ordered by a mapping $\phi$ we entirely neglect the special character of the elements, simply retaining their distinguishability and taking into account only the relations to one another in which they are placed by the ordering mapping $\phi$, then these elements are called $\textit{natural}$ ${numbers}$.  (\S 73, p. 809)
\end{quote}

\noindent The precise ontological status of these numbers won’t concern us here, but for the record they are ‘a free creation of the human mind’ (\S 73, p. 809).\footnote{Exactly what Dedekind intends by `free creation' is beside the point for our purposes, but it is a fascinating question.  See  \cite{Reck} and  \cite{MR3888849} for quite different readings.}

The so-called Peano Axioms are implicit in the definition of simply infinite: if $S$ is simply infinite, with $\phi$ from $S$ to $S$, and $s \in S$, then $A \subseteq S$, $s \in A$, and $A$ closed under $\phi$ implies $A = S$.  With the axioms in place, the usual theorems of arithmetic follow and ‘are always the same in all ordered simply infinite systems, whatever names may happen to be given to the individual elements’ (\S 73, p. 809).  

This is certainly true – the consequences of $\PA$ hold in all relevant systems – but it isn’t enough to guarantee that performing the process of abstraction described in the previous paragraph on different simply infinite systems won’t generate different natural numbers.  Apparently sensitive to this loose end, Dedekind refers us forward to his proofs that any simply infinite set is isomorphic to the numbers (\S 132) and that any set isomorphic to the numbers is simply infinite (\S 133).  Thus, the simply infinite systems form an isomorphism class (\S 34), so it doesn't matter which one he happens to abstract from in \S 73. Only then, in \S 134, does he conclude that ‘the definition of the notion of numbers given in \S 73 is fully justified’ (\S 134, p. 823).  This is the role of his famous categoricity result in the overall project of showing that the natural numbers can be characterized without appeal to intuition, in purely logical terms.

So, how does the proof go?  Averting our gaze from 
`free creations', we  sketch a direct Dedekind-style proof that any two simply infinite sets are isomorphic:  

\newtheorem{theorem}{Theorem}
\newtheorem{lemma}{Lemma}

\begin{theorem}\label{ded1} Suppose $N$ and $N'$ are simply infinite, that is, that there are $a\in N$ and $a'\in N'$, and one-one functions $\phi:N\to N$, and $\phi':N'\to N'$ such that
	
	\begin{equation} 
		N=\bigcap\{X:a\in X  \mbox{ and $X$ is a } \phi\mbox{-chain}\}
	\end{equation}
	
	and 
	
	\begin{equation}
	N'=\bigcap\{X:a'\in X  \mbox{ and $X$ is a } \phi'\mbox{-chain}\},
	\end{equation}

then $N \cong N'$.
	
\end{theorem}

\begin{proof}
	
Suppose $N$ and $N'$ are as above. To construct an isomorphism between $(N,\phi,a)$ and $(N',\phi',a')$, proceed by induction in $N$: for $b \in N$, define $\psi_b$, an isomorphism from $b$ and its predecessors into $N'$ such that $\psi_b(a)=a'$ and $\psi_b (\phi (x)) = \phi' (\psi_b (x))$ (\S 125). Then let $\psi(b)=\psi_b(b)$ (\S 126), and show, by induction on $N'$, that $\psi$ is onto $N'$.  $\psi$ is the required isomorphism (\S 132).

\end{proof}

\subsection{Does what he did accomplish what he set out to do?}\label{subsec:ded3}  It seems fair to say that Dedekind has achieved his goal of characterizing the natural numbers without appeal to geometric or temporal intuition, but from our contemporary perspective, it's impossible to ignore the freedom with which he allows `different things' to be `considered from a common point of view [and] associated in the mind' so as to form a set (§2, p. 797). Dedekind reports that the dangers of this free use of comprehension came clear to him in 1903 and troubled him so much that he was at first reluctant to allow a third edition of his little book:

\begin{quote}

When I was asked $...$ to replace the second edition of this work (which was already out of print) by a third, I had misgivings about doing so, because in the meantime doubts had arisen about the reliability of important foundations of my conception.  Even today I do not underestimate the importance, and to some extent the correctness, of these doubts.  \cite[p.  796]{02693454}

\end{quote}

\noindent Still, though he was `prevented by other tasks from completing such an investigation', Dedekind remains confident:
\begin{quote}
	My trust in the inner harmony of our logic is not $...$ shattered; I believe that a rigorous investigation of the power of the mind to create from determinate elements a new determinate, their system, that is necessarily different from each of these elements, will certainly lead to an unobjectionable formulation of the foundations of my work.  (p. 796)
\end{quote}

\noindent As we know, he turned out to be right about this. 

Of course, any way of regularizing Dedekind's reasoning to contemporary standards will inevitably distort it in some way.  Two methods appear in the literature -- one set-theoretic, one second-order -- with the latter perhaps more common.  To assess the moral of Dedekind's work for ourselves today, we look at each of these in turn.

The first is fairly direct.  Dedekind's presentation is already formulated in an early set-theoretic context; the only problem is his assumption of what looks perilously close to Unlimited Comprehension.  In his categoricity theorem, this comes into play when he forms the family of all $\phi$-chains containing $a$ ($N$ is its intersection).  In fact, though, this isn't problematic, because all those $\phi$-chains are subsets of $N$, so the relevant family can be generated by Separation from $\mathcal{P} (N)$.  On closer inspection, the relevant intersection can be formulated as ...
$$
\begin{array}{lcl}
	\{x\in N:\forall y((\forall z\in y(z\in N)\wedge a\in y
	\wedge\forall z\in y(s(z)\in y))\to x\in y)\}.
\end{array}$$

\noindent ... so only $\Pi _1^1$-Separation is needed.  Beyond that, we know today that much weaker principles suffice:  in \cite{MR3001547}, Simpson and Yokoyama show that $\mbox{WKL}_0$ (Weak K\"onig's Lemma\footnote{Every infinite binary subtree of $2^{<\omega}$ has an infinite branch.}) suffices over $\mbox{RCA}_0$ (Recursive Comprehension Axiom). This is interesting because $\mbox{WKL}_0$ is weaker than Peano arithmetic itself (i.e. $\mbox{ACA}_0$). 

In the second, more common reading, Dedekind's reasoning is cast, not set-theoretically, but as proving the categoricity of second-order Peano Arithmetic.\footnote{See, e.g., \cite[pp. 82-83]{MR1143781}, \cite[p. 155]{MR3821510}.}  The conditions on simply infinite sets are then replaced by the second-order Dedekind-Peano Axioms, and Theorem 1 becomes:

\begin{theorem}\label{ded2} Suppose $(N, S, 0)\models \PATWO$  and $(N', S', 0')\models \PATWO$, where $N$ and $N'$ are full models.\footnote{\label{full} A second-order model is \textit{full} if it interprets the second-order variables as ranging over the full power set of the domain -- as opposed to a \textit{Henkin model}, which only needs enough subsets to satisfy the axioms of second-order logic (comprehension and choice).} Then there is a bijection $f:N\to N'$ such that
	
	$$\begin{array}{l}
f(0_N)=0_{N'}\\
f(S_N(x))=S_{N'}(f(x))
\end{array}$$
	
\noindent \mbox{that is,} $f:N \cong N'$.

\end{theorem}

\noindent  Dedekind's proof easily translates to treat models of $\PATWO$ instead of simply infinite sets, but it still takes place within his set-theoretic background theory.

Though the two readings don't differ in their underlying mathematical content, the contrast between them is worth noting.  In Theorem 1, the conditions on a `simply infinite' set are stated directly:  for example, there is no $x \in N$ such that $S(x)=0$.  In Theorem 2, those same conditions appear indirectly:  in the same example, N satisfies the axiom `there is no x such that $S(x)=0$'. The latter involves semantic notions, while the former does not; as a philosopher might put it, the former uses the set-theoretic conditions, while the latter only mentions them.\footnote{For those unfamiliar with this terminology, the word `cat' is \textit{used} in the sentence: the cat is on the mat.  It's \textit{mentioned} in the sentence: the word `cat' has three letters.\label{um}}  As we'll use the terms, Theorem 1 is a \textit{weakly internal} categoricity argument and Theorem 2 is \textit{external}.\footnote{\textit{Strongly internal} or just \textit{internal} will be introduced and eventually formalized in Subsection~\ref{subsec:par2}.} Despite the apparent preference in the literature for the external reading, the weakly internal reading strikes us as considerably closer to Dedekind's own thought, but either way, the main moral we wish to draw remains:  Dedekind has successfully achieved his goal; he's shown that the natural numbers can be described without appeal to geometric or temporal intuition.

\section{Zermelo in `On boundary numbers and domains of sets' (1930)} \label{sec:zer}

\subsection{What does he set out to accomplish?} Perhaps the first thing to note about the aims of this well-known paper (\cite{zbMATH02562682}) is that Zermelo isn't out to do what many recent observers have imagined; he isn't out to introduce the iterative hierarchy, to argue that the axioms are true therein, and to use this purported fact to motivate or defend ZFC.  In fact, the paper moves in the opposite direction, beginning with a list of the axioms and arguing that any model thereof -- any `normal domain' -- can be stratified into well-ordered ranks, into $V_{\alpha}$'s.  In a subsequent report back to his funding agency on what he's done so far, Zermelo catalogs the benefits of this analysis and proposes, as a distinct continuation, the project of constructing `a set-theoretic model' \cite[p. 441]{Z1}.  A subsequent sketch (\cite{Z2}) indicates what he has in mind -- laying out a pre-theoretic iterative picture and arguing that the axioms are true there -- but his stated goal even then is to establish the consistency of the axioms, not their truth.  So we need to look a bit more closely at his intentions in 
`On boundary numbers' (\cite{zbMATH02562682}).

The only aim mentioned in the opening paragraph of the paper is the resolution of the paradoxes:

\begin{quote}
	it is this sharp distinction between the different models of the ... axiom system that allows us to resolve the `ultrafinite antinomies' \cite[p. 29(401)]{zbMATH02562682}\footnote{The first page reference is to the original publication in \textit{Fundamenta Mathematica}, the second to the English translation in \cite{MR2640544}.} 
\end{quote}

\noindent This was also an explicit aim of his original axiomatization in \cite{zbMATH02639893}:\footnote{Moore \cite[pp.\@159--160]{MR679315} points out that Zermelo intended his early axiomatization to bolster his proof of the well-ordering theorem by clearly isolating the assumptions required.  Ebbinghaus \cite[pp.\@ 76--79]{MR2301067} acknowledges this goal, but places Zermelo's early efforts squarely in the larger context of Hilbert's foundational project, where axiomatization and consistency proofs were central -- obviously a reaction to the paradoxes.  In fact, Zermelo was inclined to withhold publication of his \cite{zbMATH02639893} until he had the desired consistency proof, but Hilbert, out of concern for his young colleague's immediate career prospects, advised him not to wait.  As noted above, Zermelo continued to pursue a consistency proof in the 1930s. }    

\begin{quote}
	In solving the problem [of axiomatization] we must, on the one hand, restrict these principles sufficiently to exclude all contradictions and, on the other, take them sufficiently wide to retain all that is valuable in this theory.  \cite[p. 191(200)]{zbMATH02639893}\footnote{Page references are to the reprinting in \cite{MR2640544}, followed by the page reference in \cite{MR0209111}.}
\end{quote}

\noindent But he was clearly dissatisfied with the resolution given there:

\begin{quote}
	I have not yet been able to prove rigorously that my axioms are `consistent', though this is certainly very essential; instead I have had to confine myself to pointing out now and then that the `antinomies' discovered so far vanish one and all if the principles here proposed are adopted as a basis.  \cite[pp. 191(200-201)]{zbMATH02639893}
\end{quote}
\noindent So at least one goal of  \cite{zbMATH02562682} is to provide a broader and more explanatory treatment of the paradoxes.

It's also worth noting the second clause in the first quotation above from \cite{zbMATH02639893}:  `sufficiently wide to retain all that is valuable in this theory'.  Zermelo's eye is always on the mathematical powers of his assumptions;  in his \cite{zbMATH02643155}, for example -- the `new proof' of well-ordering, companion to \cite{zbMATH02639893} --  he defends the axiom of choice on the basis of its theoretical benefits:

\begin{quote}
	Principles must be judged from the point of view of science [i.e., mathematics], and not science from the point of view of principles fixed once and for all.  \cite[p. 135(189)]{zbMATH02643155}
\end{quote}

\noindent Zermelo takes it as obvious that there is much of mathematical value in infinitary set theory and that axioms can be properly assessed in terms of the mathematical achievements they enable.  (These are 
so-called extrinsic justifications, which Zermelo pioneered.\footnote{For the record, extrinsic considerations weigh  the mathematical benefits of a claim, as opposed to its intrinsic obviousness or connection to the concept of set.  See, e.g., \cite{MR2779203}.\label{intext}}) So we should also expect some discussion of the mathematical consequences of the axioms he considers.\\

\subsection{What does he actually do?}  Like Dedekind, Zermelo works in a pre-theoretic background theory of sets,\footnote{Once again (cf. footnote \ref{Dede}), we set aside question about the metaphysical nature of these sets.} but unlike Dedekind, he doesn't make this explicit.  (The content of that background theory is inferred from his argumentation in what follows.)  He begins, instead, by listing the axioms of the theory he'll be studying.  Extensionality, Separation, Pairing, Power Set, and Union carry over from \cite{zbMATH02639893}.  Infinity is omitted `because it does not belong to ``general" set theory' \cite[p. 30(403)]{zbMATH02562682}, an omission whose rationale emerges only later in the paper.\footnote{See footnote \ref{inf}.}  Choice is now regarded, not as a set-theoretic axiom, but as `a general logical principle upon which our entire investigation is based' (ibid., p. 31 (405)), giving us our first glimpse of what's true in Zermelo's implicit background theory. Finally, the explicit list is supplemented by Replacement,\footnote{Zermelo notes the redundancy this introduces.}  attributed to Fraenkel, and Foundation.  Of course, the property in Separation and the function in Replacement are to be understood as `unrestricted' or `arbitrary', what we now think of as second-order, and Kanamori points out that Foundation, too, must be second-order in the absence of Infinity.\footnote{See \cite[pp. 393-394]{km}.\label{konf}}  This system Zermelo calls $\ZF'$.  

Little is said in defense of this selection of axioms.  This is understandable for the carry-over axioms, as they were treated in  \cite{zbMATH02643155} and \cite{zbMATH02639893}, and perhaps it was fair to assume that the case for Replacement has already been aired in the literature, for example, by Fraenkel and von Neumann.\footnote{Though see footnote \ref{vN}.}  The one comment Zermelo does make concerns the new axiom of Foundation:

\begin{quote}
	This last axiom, which excludes all `circular' sets, all `self-membered' sets in particular, and all `rootless' sets in general, has always been satisfied in all practical applications of set theory, and, hence, does not result in an essential restriction of the theory for the time being.\footnote{It's not clear to us what work Zermelo's many quotation marks are intended to do here and elsewhere, so we pass them by without further comment.}  \cite[ p. 31(403)]{zbMATH02562682}
\end{quote}
 
\noindent So Foundation is benign, at least as far as currently understood, and perhaps there is some intuitive appeal to ruling out `pathological' sets, but the full case in its defense only comes later (see Subsection~\ref{subsec:zer3}).

Zermelo then introduces the central object of his study:  

\begin{quote}
We call `\textit{normal domain}' a domain of `sets' and `urelements' that satisfies our `$\ZF'$-system' with regard to the `basic relation' $a \in b$.  \cite[p. 31(405)]{zbMATH02562682}
\end{quote}

\noindent  To the contemporary ear, this sounds like a model satisfying the theory, but Zermelo isn't working with a syntax/semantics distinction.\footnote{Cf. \cite{MR2301067}, pp. 182-183, describing Zermelo's discussion of `definite properties' in 1929: `there is no sharp distinction between language and meaning.  In fact, Zermelo will never be clear on this point.  His negligence will have a major impact in later discussions with Gödel.  The separation of syntax and semantics or – more adequately – its methodological control became the watershed that separated the area of the ``classic" researchers such as Zermelo and Fraenkel from the domain of the ``new" foundations as developed by younger researchers, among them Gödel, Skolem, and von Neumann'.}  To begin with, `$\in$' isn't to be reinterpreted: the $\in$ of the normal domain is the $\in$ of the background theory.\footnote{From a modern point of view, we can see that if $(M, R) \models \ZF'$, which includes second-order Foundation, then it's isomorphic to the $\in$-model generated by a Mostowski collapse.  So the fact that Zermelo doesn't entertain such reinterpretation is ultimately less significant than it might seem.} The same is true for the second-order quantifiers:  $N \models$ Separation$^2$ (where Separation$^2$ is second-order Separation) allows for sparse Henkin-style interpretations of the second-order quantifier; Zermelo claims, instead, that for every property $P$ and every set $x \in N$, there is a set $y \in N$ containing just the members of $x$ with $P$, where the relevant properties include all subsets of $N$ as understood in the background theory. Or, to take another example,  for $N$ to satisfy Power Set in Zermelo's sense is for $N$ to include $\mathcal{P}(a)$ for every $a \in N$, where that's the real $\mathcal{P}(a)$, as judged from the perspective of the background theory. Zermelo's normal domains are what we'd call \emph{full} models of the second-order axioms.\footnote{Recall footnote \ref{full}.}
 
With the axioms and the notion of normal domain in place, Zermelo begins his study by introducing the von Neumann  ordinals, starting not from the empty set but from an arbitrary urelement:  $u, \{u\}, \{u, \{u\}\}, ... $. Working in the background theory here, he's implicitly relying on Pairing.  Moments later, Union is applied to define the successor of the $\alpha$th $u$-based ordinal, $g_\alpha$, as $g_\alpha \cup \{g_\alpha\}$, but more importantly, this happens in the course of a definition by recursion -- revealing that versions of all the $\ZF'$ axioms, including Replacement, are present in the background theory (with the possible exception of Foundation, see Subsection~\ref{subsec:zer3}). The collection of all ordinals in a given normal domain, $N$, is itself an ordinal, though not an ordinal in $N$.  This ordinal is called the `characteristic' or `boundary number' of $N$, and Zermelo proves that it must be a (strongly) inaccessible cardinal.

From here, the path to the theorem of interest is direct.  It takes one lemma:

\begin{lemma}	
If $N$ is a normal domain, $N' \subseteq N$ with the same urelements, and
\begin{enumerate}
	\item $N'$ is transitive, and
	\item $A \subseteq N'$ and $A \in N \rightarrow A \in N'$,

\end{enumerate}

\noindent then $N'=N$.
\end{lemma}

\noindent This is where Foundation (in $N$) makes its essential appearance (in a proof by contradiction from the assumption that $N-N'$ is nonempty).\footnote{This is also where Kanamori \cite[pp. 393-394]{km} locates the second-order axiom. Recall footnote \ref{konf}.} 

Now suppose $N$ is a normal domain with boundary number $\kappa$.   Then
let 
	$$\begin{array}{lcl}
	N_0&=&\mbox{ the urelements of $N$},\\		
	N_{\alpha+1} & =&N_{\alpha} \cup \mathcal{P}(N_{\alpha}),\\	
	N_{\nu}& =&\bigcup_{\beta < \nu} N_{\beta}, \mbox{ for limit $\nu<\kappa$}.

\end{array}$$	
\noindent Given this familiar stratification on $N$ into ranks, all that's left to show is that it exhausts $N$:

\begin{theorem}
	 $\bigcup_{\alpha < \kappa} N_{\alpha} = N$.
\end{theorem}

\noindent  $\bigcup_{\alpha < \kappa} N_{\alpha} \subseteq N$, and $\bigcup_{\alpha < \kappa} N_{\alpha}$ and $N$ have the same urelements by definition, so the lemma applies.  The proof naturally uses the fact that $\kappa$ is inaccessible.

This sets up the central theorem:

\begin{theorem}\label{ZermeloCatFull}
	If $N$ and $N'$ are normal domains with the same boundary number and equinumerous urelements, then $N \cong N'$.
\end{theorem}

\noindent The proof begins from a bijection of the urelements of $N$ to the urelements of $N'$ and builds an isomorphism from $N$ to $N'$ in stages for each rank $N_{\alpha}$ of $N$ in a way that's now familiar to observers of set-theoretic categoricity arguments. 

Several observations are in order about Zermelo's line of thought here.  To begin with, though he himself uses the term, this isn't really a categoricity theorem in the sense in use today:  `$\in$' isn't given different interpretations in $N$ and $N'$.  In fact, if the urelements are eliminated, the result reduces to the elementary fact, familiar from Set Theory 101, that any $\in$-model of $\ZF^{\hspace{1pt}2}$ is a $V_{\kappa}$, for some inaccessible $\kappa$. If you asked Zermelo, he'd no doubt insist that this is proved in (what we regard as) a second-order background theory, but all that actually  matters is that $N$ and $N'$ satisfy the second-order $\ZF'$ -- ordinary, first-order $\ZFC$ is enough in the background; in fact, ZFC restricted to $\Sigma_2$-separation and $\Sigma_2$-replacement will do (see \cite{MR505489}). 

So, ironically, what's made Zermelo's proof a touchstone in the contemporary literature on categoricity arguments is the presence of those urelements.  Because those of $N$ and $N'$ are only equinumerous, not identical, the isomorphism isn't the identity -- it has to be constructed rank by rank, from the initial bijection.  Though Zermelo doesn't reinterpret `$\in$', he could have done so with essentially the same proof, which is why many commentators\footnote{See, e.g., \cite[pp. 66-67]{MR1029277}, \cite[pp. 178-179]{MR3821510}} take him to have proved the following theorem about second-order $\ZFC$ 
($\ZFCTWO$):

\begin{theorem}\label{ZermeloCat}
If $N \models \ZFCTWO$ and $N' \models \ZFCTWO$, where $N$ and $N'$ are full models with the same height
and equinumerous urelements, then $N \cong N'$.  
\end{theorem}

\noindent Actually, as with Dedekind, it's probably more faithful to see him as having proved Theorem \ref{ZermeloCatFull} in his background theory of sets.  In any case, from there he goes on to observe that if two normal domains have the same urelements, then one is isomorphic to an initial segment of the other, and that normal domains starting from the same single urelement (as $\emptyset$) stack neatly in the order of their boundary numbers.\\

\subsection{Does what he did accomplish what he set out to do?} \label{subsec:zer3} 

To answer this question, we need to follow Zermelo's line of thought a bit further.  Working in his background theory,\footnote{Setting urelements aside.} he's shown that there's effectively one normal domain of each height, that each of these breaks down into ranks -- essentially the theorem $V=\bigcup_{\alpha \in Ord}V_{\alpha}$,\footnote{This notation just means: $\forall x \exists \alpha (x \in V_{\alpha})$.} --
and that they form a hierarchy matching that of their boundary numbers.  These mathematical facts suggest an image of the universe of sets as an endless series of neatly stratified domains.  This image quickly grew into a compelling intuitive picture of emphatic richness,\footnote{This includes the combinatorial ideas of  
\cite{zbMATH02531279}.}, \footnote{\cite{Boolos1971-BOOTIC} is an outlier.  His austere version of the Iterative Conception doesn't include replacement -- which many would regard as the first affirmation, after the axiom of infinity, that the sequence of ordinals is `very long' -- or choice -- which many would regard as an affirmation that the power set is `very thick' (see  \cite{zbMATH02531279}, p. 260).} the Iterative Conception:

\begin{quote}
	We believe that the collection of all ordinals is very `long' and that each power set (of an infinite set) is very `thick'. \cite[p. 553]{wang}
\end{quote}

\noindent This picture suggests new methods like reflection and maximality arguments, new axioms like large cardinal and forcing axioms, along with all manner of elaborations and analyses.\footnote{In a striking example,  \cite[\S 18.3]{MR4352351}, argues that the iterative picture `isn't just a nice extra feature that allows set theory lecturers to draw pretty pictures on the board', rather `it is vital' (p. 402):  `Humans are adapted to build complex tools, which means building up objects from component parts, themselves built from yet simpler components' (p. 403).} Zermelo would insist on judging each resulting mathematical proposal by its extrinsic mathematical merits, and there's little doubt that he'd regard the Iterative Conception in this heuristic use as an impressive source of productive mathematical ideas.

All very good, one might say, but what if there are no boundary numbers? Zermelo addresses this question directly in the final section of the paper, beginning with the observation that the hereditarily finite sets form a normal domain that even the intuitionist would accept.\footnote{\label{inf}Now we see why Zermelo leaves Infinity off the list $\ZF'$.}  The trouble with stopping here, he goes on to note, is that it doesn't permit `Cantorian set theory', and thus would not, in words going back to 1908, `retain all that is valuable in this theory' \cite[p. 192(200)]{zbMATH02639893}.  This was Zermelo's extrinsic justification back then, and it remains so now:  he includes the existence of $\omega$, the boundary number of the normal domain of hereditarity finite sets, and with it the larger normal domain whose boundary number is the first uncountable inaccessible.  Only a domain that satisfies $\ZF'+$Infinity deserves to be called `Cantorian'.

Notice that Zermelo isn't advocating that we add Infinity to $\ZF'$; he's advocating that we add it to the background theory, to the theory that determines which normal domains there are and what they're like.   In that background theory, he claims the right to Cantorian set theory, rejecting the intuitionist's effort to restrict set-theoretic inquiry, while at the same time preserving the normal domain $V_{\omega}$, where constructivist studies with their own extrinsic virtues can still be carried out.\footnote{We abuse notation by using $V_{\omega}$ for the hereditarily finite sets:  it suggests the existence of $\omega$, which, of course, is not present in that normal domain.  We hope that this is understood, that the usage doesn't mislead, here and elsewhere, where $V_{\kappa}$ is used for the normal domain whose boundary number is $\kappa$.} Thus he takes the non-categoricity of $\ZF'$ as an advantage, not a disadvantage:  `set theory as a \textit{science} must ... be developed in greatest generality' \cite[p. 45(427-428)]{zbMATH02562682}.  The italicized `science' is telling. Here, as in his earlier writings, it signifies his commitment to what he sees as productive mathematics.  `The comparative investigation of individual \textit{models}' (ibid.) is then one aspect of this productive project.

Under the banner of this generality, Zermelo also resists efforts to limit set theory to `the \textit{lowest} infinitistic domain', to $V_{\kappa_{1}}$, for $\kappa_{1}$ the first inaccessible.  An axiom to that effect would secure categoricity, but it would not `concern set theory \textit{as such}'; it would `only characterize the special \textit{model} chosen by the respective author' \cite[p. 45(427)]{zbMATH02562682}.  The better course is to acknowledge that the ordinals of $V_{\kappa_{1}}$ constitute a boundary number greater than $\omega$ -- namely, $\kappa_1$ -- and that this uncountable boundary number itself exists as a set in a normal domain larger than $V_{\kappa_{1}}$.  

\begin{quote}
	It is of course not possible to `prove', that is, to deduce from the general $\ZF'$-axioms, either its existence or non-existence, simply because, for instance, the boundary number $\omega$, even though it exists in the `Cantorian' domain [$V_{\kappa_{1}}$], does not exist in the `finitistic' domain [$V_{\omega}$], because, in other words, the question receives different answer in different `models' of set theory, and is thus not decided merely by the axioms alone.  \cite[p. 45(427)]{zbMATH02562682}
\end{quote}

\noindent So, in the interests of productive science, this and further boundary numbers should be posited in the background theory:\footnote{Zermelo also provides a somewhat inscrutable direct argument – what might be regarded as an intrinsic argument (cf. footnote \ref{intext}) – for the existence of inaccessibles based on the assumption that `\textit{every categorically determined domain can also be conceived of as a ``set"} … an element of a (suitably chosen) normal domain' \cite[p. 46 (429)]{zbMATH02562682}.  Given that the categorically determined domains are the $V_{\kappa}$s, this comes to assuming that every $V_{\kappa}$ is a set in a larger $V_{\kappa'}$.}  

\begin{quote}
	We must postulate the \textit{existence of an unlimited sequence of boundary numbers} as a new axiom for the `meta-theory of sets' … To the unlimited series of Cantorian ordinal numbers there corresponds a likewise unlimited … series of essentially different set-theoretic models in each of which the entire classical theory finds its expression.    \cite[p. 47(429-431)]{zbMATH02562682}.
\end{quote}

\noindent The upshot of this line of thought is (1) an informal background theory that essentially comes to  $\ZFC+\mbox{I}$  (where `I' asserts the existence of arbitrarily large inaccessibles), (2) the development in that background theory of detailed development of full models of the second-order theory $\ZF'$, and (3) the heuristic, intuitive picture of the Iterative Conception.

To return, at last, to our leading question, has Zermelo's categoricity theorem (and the developments he bases on it) accomplished what he set out to accomplish?  His explicit goal was to resolve the paradoxes, which he now dismisses in one characteristically combative sentence:

\begin{quote}
	The `ultrafinite antinomies of set theory', to which scientific reactionaries ... appeal in their fight against set theory with such eager passion, are only apparent `contradictions', due only to a confusion between \textit{set theory itself}, which is non-categorically determined by its axioms, and the individual \textit{models} representing it:  what in one model appears as `ultrafinite ... ', is already a fully valid `set' ... in the next higher model and, in turn, serves itself as the bed-stone in the construction of the new domain.  \cite[p. 47(429-431)]{zbMATH02562682}

\end{quote}

\noindent The class of all sets (or all ordinals) of one normal domain is just an ordinary set in the next -- fine -- but this doesn't solve the problem for the background theory:  why is there no set of all sets (or all ordinals) there?  Zermelo doesn't answer this question explicitly, but implicitly, he has a response of a piece with one offered by some observers to this day:  there's no set of all sets because every set occurs in some normal domain, and no normal domain contains all sets.\footnote{He says something very like this in \cite[p. 439]{Z1}: `Every normal domain is itself a ``set" in all higher domains, but there is no highest normal domain containing all \textit{all} normal domains as sets'.}  Today we might simply say there's no set of all sets because every set first appears at some stage of the iterative hierarchy, and there's no stage at which all sets are available to be collected.  Of course, there's no contemporary consensus on the proper treatment of the paradoxes, but Zermelo's approach remains among the live options.

Still, in truth, it would be a grave mistake to count this approach to the paradoxes as the sole payoff to Zermelo's categoricity theorems. In that retrospective report to the funding agency, Zermelo describes his goal more broadly than in `On boundary numbers' itself:
\begin{quote}
	I posed for myself the decisive preliminary question … How does a `domain' of `sets' and `urelements' have to be constituted so that it satisfies the `general' axioms of set theory?  Is our axiom system `categorical' or are there a multitude of essentially different `set-theoretic' models?  \cite[p. 437]{Z1}
	
\end{quote}

\noindent Giving answers to these questions is obviously a leading achievement of his \cite{zbMATH02562682}, but in \cite{Z1}, Zermelo goes further.  We noted earlier that the only defense of Foundation in \cite{zbMATH02562682} is the observation that it `does not result in an essential restriction of the theory for the time being' \cite[p. 31(403)]{zbMATH02562682}; now, in \cite{Z1}, he explicitly acknowledges its indispensable role in the results of that paper:\footnote{Oddly, even in \cite{Z1}, he doesn't acknowledge the decisive role of Replacement in his \cite{zbMATH02562682}.  He remarks that it was Hausdorff's introduction of cofinality `that made possible a fruitful application of the new axiom' \cite[p. 435]{Z1}, omitting reference to von Neumann's use of Replacement for Transfinite Recursion, on which  \cite{zbMATH02562682} obviously draws (cf. \cite[p. 432]{km1}).\label{vN}}

\begin{quote}
	In order to tackle [these goals] successfully, however, I first had to supplement the `Zermelo-Fraenkel axiom system' by adding a further axiom, the `foundation axiom' ... By using the new axiom it was possible to carry out a decomposition into layers … of a `normal domain' ... and to answer the decisive main question in the `isomorphism theorems'.  \cite[p.  437]{Z1}
\end{quote}

\noindent The fact that Foundation enables the clarifying and fruitful breakdown of $V$ into the $V_{\alpha}$s constitutes a strong extrinsic argument in its favor, both as an axiom of $\ZF'$ and as an assumption of the background set theory.

Finally, beyond a viable candidate response to the paradoxes, beyond the clarification and insights arising from the theorem $V=\bigcup_{\alpha \in Ord}V_{\alpha}$ and its surroundings, beyond persuasive defense of the new axioms of Foundation and Inaccessibles, `On boundary numbers' introduced the Iterative Conception, whose mathematical fruitfulness for the future of set theory can hardly be overstated.  In the end, then, Zermelo accomplished more than he set out to do – and ultimately more than he could have realized at the time -- so this application of categoricity arguments must be counted as a resounding success.

\section{Kreisel in `Informal rigor and incompleteness proofs' (1967) and `Two notes on the foundations of set theory'(1969)}\label{sec:kre}

Kreisel's philosophical discussions of categoricity arguments in \cite{Kreisel1969-KRETNO-2} and \cite{Kreisel1967-KREIRA} have continued to influence contemporary thought on the significance of such results, so we pause to ask what conclusions he intends to draw and to what extent he succeeds.  (Since he presents no new mathematics, we dispense with our usual division into subsections.)

Looking back at `On boundary numbers', Zermelo notes in passing one consequence of its categoricity results:

\begin{quote}
	From this it already follows, among other things, that \textit{Cantor's} … conjecture … does \textit{not} depend on the choice of the model, but that it is decided (as true or as false) once and for all by means of our axiom system.  \cite[p.  437]{Z1}\footnote{Zermelo is actually referring to GCH, but here we keep the focus on CH.  The claim is untenable, if not false, for GCH.}
\end{quote}

\noindent Understood as a theorem in Zermelo's background set theory, the claim is that CH has the same truth value in every normal domain,\footnote{Really, every normal domain with the same number of urelements.  Our focus here is on `unit' normal domains, that is, normal domains with one urelement playing the role of $\emptyset$, so from now on, we ignore these niceties.} but Kreisel reformulates it as a fact of second-order logic:

\begin{center}
	$\ZFCTWO \models CH$ or $\ZFCTWO \models \neg CH$
\end{center}

\noindent What conclusion(s) does he draw from this? 

Kreisel is often taken to have argued that this theorem of second-order logic shows that CH has a determinate truth value (even though, sadly, we haven't been able to figure out what that is).\footnote{See, e.g., 
\cite[p. 180]{MR3821510}.}  In response, observers point out that if we're concerned that CH may not have a determinate truth value because we worry that the power set operation isn't fully determinate, then we are, or should be, just as concerned about the determinateness of the second-order quantifiers.\footnote{Weston was first to note this in print (see \cite{MR480033}). Button and Walsh \cite[p. 158, note 17]{MR3821510}  list several others since. (Cf. Quine's assessment that higher-order logic is `set theory in sheep's clothing' \cite[p. 66]{MR0469684}.)  In contrast, Zermelo's set-theoretic proof that that all normal domains agree on CH might as well be carried out in first-order ZFC.}   The resulting assessment is that Kreisel was wrong and no progress on CH has been made.

The surprise is that Kreisel is well aware of this problem, as he remarks in a long endnote:

\begin{quote}
	we still have the following \textit{simple-minded puzzle}: is it not circular to use second order notions which involve the concept of set (or: subset) in axiomatizations of this concept?  \cite[p. 111, endnote 2]{Kreisel1969-KRETNO-2}
\end{quote}

\noindent  This raises serious questions about what conclusions, if any, Kreisel is inclined to draw about the determinateness of CH -- there's much to ponder in this endnote (see below) -- but it's reasonable to assume that his central arguments appear, not there, but in the main text.  

In fact, both articles are entirely explicit about his goal in appealing to the categoricity of $\ZFCTWO$.  Writing in the late 1960s, Kreisel was keen to isolate and articulate the significance of Cohen's recent independence results, as were many others in the set-theoretic and larger intellectual communities.  (Kreisel also intended to improve `logical hygiene'  \cite[p. 109]{Kreisel1969-KRETNO-2} by rebutting various popular misreadings.)  Kreisel has a lot to say about this, some of it admittedly inconclusive, but he emphasizes one clear point in light of Zermelo's result, namely, the importance of 

\begin{quote}
	\textit{Distinctions} formulated in terms of higher order consequence.  In contrast to the example of CH ..., Fraenkel's replacement axiom is not decided by Zermelo's Axioms (because [Zermelo's axioms plus Infinity are] satisfied by [$V_{\omega + \omega}$] and Fraenkel's axiom is not); in particular it is independent of Zermelo's second order axioms while by Cohen's proof, CH is only independent of the \textit{first order schema} (associated with the axioms) of Zermelo-Fraenkel.  \cite[pp. 150-151]{Kreisel1967-KREIRA}
\end{quote}

\noindent Returning to this point from \cite{Kreisel1967-KREIRA} in \cite{Kreisel1969-KRETNO-2}, he puts it this way:

\begin{quote}
	The continuum hypothesis $CH$ is (provably) \textit{not} independent of the full (second order) version of {\sc Zermelo}'s axioms; we know this much without knowing which way $CH$ is decided; the example is not empty since, for instance, the replacement axiom \textit{is} second-order independent.  Needless to say, from the point of view of our present \textit{knowledge} of the hierarchy of sets, it is of great interest to establish \textit{what can and what cannot be decided }from the first order schemas or, as we might put it, \textit{from our present analysis of the full axioms}.  \cite[p. 107]{Kreisel1969-KRETNO-2}
\end{quote}

\noindent In addition to its inherent importance, this demonstration of the distinction between first- and second-order independence also facilitates one of Kreisel's exercises in `logical hygiene':  the popular analogy with the independence of the parallel postulate is misleading because the parallel postulate is independent of second-order geometry; the parallel postulate `corresponds to Fraenkel's axiom, not CH' \cite[p. 151]{Kreisel1967-KREIRA}.\footnote{See also \cite[p. 111, endnote 2.]{Kreisel1969-KRETNO-2}.}

But now, what about the determinateness of CH?  In the endnote quoted above, he's considering the \emph{standard argument} so often attributed to him, an argument to the conclusion -- `subset of the reals' is determinate -- from the premise -- a second-order quantifier over the reals is determinate. If the familiar reading of this paper were correct, we'd expect Kreisel to launch a direct defense against the circularity objection, but he doesn't.  What he says is:

\begin{quote}
	Sure, it would be circular if one were looking for a \textit{reduction}, a definition of this concept [of set (or: subset)] in, say, more elementary terms.  \cite[p. 111, endnote 2]{Kreisel1969-KRETNO-2}
\end{quote}
	 
\noindent Well, yes, something like that is precisely what we're looking for in the standard argument:  we're supposed to have confidence that the second-order quantifier is determinate -- it's just logic, don't you see? -- and that confidence is supposed to be transferred to the notion of subset.\footnote{For comparison, we might be suspicious of the semantic notion of first-order consequence and find that suspicion relieved by the discovery that it's equivalent with a syntactic one (Kreisel's own example).}  A few lines later, he continues -- `to suspect a \textit{vicious} circle we should have to have independent reasons, such as the ambiguities mentioned at the beginning' (ibid.) of the paper, namely concerns about `the basic notion of set' \cite[p. 94]{Kreisel1969-KRETNO-2}.  Well, yes, again, the standard objection points out that someone who begins with worries about the concept of set won't be helped by the standard argument.

So it seems that such defense as Kreisel offers for the determinateness of CH isn't the standard argument, but a direct case against a few of the reasons one might have to worry about the notion of subset, and hence about CH, in the first place.  The two reasons he considers are based on predicativism and finitism, respectively, and he argues that they're both `quite inconclusive' \cite[p. 96]{Kreisel1969-KRETNO-2}.  But surely there are other, more compelling grounds on which to doubt the determinacy of the notion of subset.  

On this score, Kreisel himself doesn't seem entirely confident.  Speaking just epistemologically, about our ability to determine the status of CH, he writes:  

\begin{quote}
	There are a lot of subsets of $\omega$, there are a lot of 1-1 mappings from $\omega$ to an initial segment of the ordinals; we certainly do not \textit{know} ... any listing of the subsets of $\omega$ by means of ordinals.  So why should we expect to \textit{know} the answer to this particular `simple' matter?  \cite[p. 109]{Kreisel1969-KRETNO-2}
	
	Doesn't one simply have an overwhelmingly strong impression that though we don't know how to decide CH, any decision would involve considerations of a quite \textit{different character} from those which have led to existing axioms?  (Ibid., p. 108)
\end{quote}

\noindent He observes that these new considerations needn't be ``arbitrary" or even non-mathematical' (ibid.).  

 He then moves from an epistemic worry that we don't know how to find the truth value of CH to a worry at issue in the standard argument, the worry that its truth value might be metaphysically indeterminate. The response he considers is based on what he calls `realism',\footnote{\label{realism}He leaves many specifics of what he means by this far from univocal term to his reader's imagination.} which he criticizes on exactly this point earlier in the paper:

\begin{quote}
	there \textit{is} or, at least, seems to be a genuine defect in the realist position generally and, in particular, in the unanalyzed use of the power set operation.  \cite[p. 97]{Kreisel1969-KRETNO-2}
\end{quote}

\noindent Given that he links the determinateness of power set with realism and that he takes this aspect of realism to be problematic, it's unlikely that vanquishing predicativism and finitism has left him immune from other worries about the notion of subset -- and we've seen that he doesn't expect such worries to evaporate in the face of the standard argument.

To better understand Kreisel's attitude here, consider another criticism he lodges against realism, this one more methodological than epistemological or metaphysical:

\begin{quote}
	\textit{By the second order decidability of CH, [realism] demonstratively does not provide a framework within which one can discuss the impression} [that considerations of a quite different character would be needed to settle it].  (\cite[p. 108]{Kreisel1969-KRETNO-2}, emphasis in the original)
\end{quote}

\noindent It isn't immediately obvious why he thinks realism can't provide a framework for discussing the need for new methods, but there are clues.  After dismissing the objections from predicativism and finitism, Kreisel tries out a realistic posture himself and examines `the cumulative type structure'.  From that perspective, he finds that `we have very good evidence for some of the axioms of infinity', for example, `when one recognizes that [a candidate] is a consequence of a second-order reflection principle then one has \textit{found} good evidence' \cite[pp. 98-99]{Kreisel1969-KRETNO-2}.  Why is this evidence good?  Because `the validity of the reflection principle becomes apparent on closer analysis of the hierarchy', because `in terms of \textit{knowledge} of the unbounded hierarchy, reflection principles are quite evident' (ibid.) -- and he goes on to give an intrinsic argument based on the unattainability of $V$.  So perhaps he thinks, as it not uncommon, that the only justificatory methods available to the realist are intrinsic, and given that he also thinks progress on CH will require methods of `a quite different character from those which have led to existing axioms' (ibid., p. 108), it would follow that realism hasn't the means to address the fundamental challenge posed by the phenomenon of its independence.

As an alternative to realism, he also considers what he calls `formalism'.\footnote{As with `realism' (see footnote \ref{realism}), the particulars of this position, it's relation to other views of the same name, are left unexplored.}  Presumably a formalist eschews metaphysics, so isn't bound by the idea that only intrinsic considerations can be trusted to track the contours of some abstract subject matter. This leaves open the possibility of new methods, but unfortunately, `the formalist position does not normally \textit{attempt} an analysis of the choice of system at all' \cite[p. 108]{Kreisel1969-KRETNO-2}.  Notice that this isn't a principled argument that a formalist couldn't take this on, just a complaint that his contemporary formalists don't -- the possibility remains open of a version of formalism that does.  Such a formalism might appeal, for example, to mathematical effectiveness in its justifications, to extrinsic justifications -- following Zermelo.\footnote{The first author, predictably, thinks of the Arealism of her \cite{MR2779203} or the Enhanced If-thenism of her \cite{maddy:toappear}.   Kreisel comes close to this theme in a discussion of `pragmatism':  he grants it a `superficial plausibility', but complains that it discourages `work on ... intuitive notions' \cite[pp. 140-141]{Kreisel1967-KREIRA}.  But this needn't be true; e.g., we've already noted the immense heuristic value of Zermelo's Iterative Conception (cf., e.g., \cite[p. 136]{MR2779203}).  Kreisel also worries that a philosopher with such views will be left with nothing left to do, but the first author obviously disagrees.}  In fact, Kreisel sees encouraging  exploration of new methods as an `attempt to use [independence results] for some objectively significant purpose' \cite[p. 110]{Kreisel1969-KRETNO-2}.\footnote{\cite{Kreisel1969-KRETNO-2} closes with an implicit appeal for more such exploration:  `Here it must be admitted that, if CH, resp. $\neg$ CH is to be distinguished from other axioms of set theory by the \textit{kind} of consideration needed to decide it, the crudity of current discussions is certainly not surprising.  For though foundations have made striking progress in the precise analysis of \textit{notions}, the analysis of reasons has been much less successful; and what \textit{has} been done is little known' \cite[p. 110]{Kreisel1969-KRETNO-2}.}

Finally, one last observation from the fascinating endnote 2 of \cite{Kreisel1969-KRETNO-2}. We've seen that Kreisel regards the standard argument from the determinateness of second-order logic to that of CH as potentially viciously circular, but he also locates a benign circle in the vicinity.  We might, he suggests, simply `use a concept in order to \textit{state the facts} about it' \cite[p. 111, endnote 2]{Kreisel1969-KRETNO-2} and illuminate or develop a concept in this way.  Consider, for example, the use of logical particles to advance an account of the satisfaction conditions for sentences containing the logical particles.  He doesn't elaborate, but the idea must be that we better understand the subset operation and the second-order quantifier when we appreciate their interconnection in the case of CH.

 In sum, then, Kreisel sees Zermelo's categoricity arguments as accomplishing at least two things. Since his overarching goal is to probe the significance of the independence results, his main focus is on the first of these: revealing that CH presents a new kind of independence -- first-order but not second-order -- which differentiates it from familiar cases like Replacement, large cardinals, or the parallel postulate.  Along the way, in his analysis of the standard argument for the determinacy of CH, he observes that the categoricity results also provide a sort of elucidation or illumination of the concept of subset, by bringing out its interconnections with second-order quantification.  As for the standard argument as intended, he doesn't appear to take the threat of circularity to be fully dispelled, but the determinateness of CH wasn't his target in the first place.  At his actual goal -- elucidating the independence phenomenon -- he succeeds.

\section{Parsons in `The uniqueness of the natural numbers'  (1990) and `Mathematical induction' (2008)}\label{sec:par}

\subsection{What does he set out to accomplish?}

In Chapter 8 of his book, \textit{Mathematical Thought and its Objects} \cite{MR2381345} -- especially `The problem of the uniqueness of the number structure’ (\S 48) and ‘Uniqueness and communication’ (\S 49) -- Charles Parsons develops and refines the line of thought in his 1990 paper, `The uniqueness of the natural numbers' \cite{10.2307/23350653}.  His topic isn’t the uniqueness of individual natural numbers, like 3; as a structuralist, Parsons holds that various objects, perhaps various sets, for example, could play the role of 3 in the natural number structure.  Rather, his concern is whether

\begin{quote}
	the natural numbers are at least determinate up to isomorphism: If two structures answer equally well to our conception of the sequence of natural numbers, they are isomorphic.  \cite[ §48, p. 272]{MR2381345}
\end{quote}

\noindent A claim like this would be vacuous if there were no structure answering to our conception, but Parsons argues, like Dedekind  before him, that there is at least one such (see below).  As he understands it, ‘the initial, intuitive ground’ \cite[ §48, p. 272]{MR2381345} for concern is that our conception might be vague somehow, raising the possibility that it could be made precise in different ways.

The conception Parsons has in mind is given by collection of rules stating:\\

$$
\begin{array}{c}
N(0) \\ 
\\
N(x)\rightarrow N(S(x))\\
\\
S(0) \neq 0\\
\\
\infer{x=y}{S(x)=S(y)}\\
\\
\infer{P(t)}{P(0) & \infer*{P(S(x))}{P(x)} & N(t)}
\end{array}
$$

\noindent for any predicate $P$ and term $t$.\footnote{See \cite[\S \S 31, 47]{MR2381345}}  An `inescapable vagueness' \cite[\S 47, p. 270]{MR2381345} arises because this characterization involves `induction as an inference that could be made with any well-defined predicate, without the prospect of specifying exactly what the range of such predicates is' (ibid., \S 48, pp. 272-273).\footnote{Parsons often discusses a purported threat of vagueness rooted in the existence of non-standard models, but ultimately concludes that it is `quite unconvincing' \cite[\S 48, p. 279]{MR2381345}.  We happily leave this aside, along with the difficult question of how a structure answering to a conception -- e.g., the Hilbertian stroke types described below -- relates to a model of set-theoretic model theory. The `modelism' of \cite{MR3821510} appears to solve this problem by fiat: `mathematical structures, as discussed informally by mathematicians, are best explicated by … a class of isomorphic models' (p. 38).}  We call this understanding of a predicate variable `open-schematic'.  

The potential problem comes into sharper focus when we attend to Parsons's example of a structure that answers to our conception.\footnote{See \cite[\S \S 28-29]{MR2381345}.  Here and elsewhere, we use Parsons's own locution -- `answers to our conception' -- because we're unsure what would count as a fair paraphrase.    More generally, our goal in this subsection is to describe Parsons's view, not to defend or endorse it.}  He starts from Hilbert \cite{ontheinfinite}:

\begin{quote}
	Following Hilbert we begin by considering the `syntax' of a `language' with a single basic symbol `$|$' (stroke), whose well-formed expressions are just arbitrary strings containing just this symbol, i.e., $|$, $||$, $|||$, … \cite[§28, p. 159]{MR2381345}
\end{quote}

\noindent On Parsons's account, we perceive these concrete tokens of strings.  The string types he calls `quasi-concrete': they are abstract and acausal, but also typically intuited in the process of perceiving their tokens (assuming the perceiver enjoys the modest conceptual resources required).\footnote{Again following Hilbert, Parsons \cite[\S 28, p. 161]{MR2381345} says that `intuition of a type [is] \textit{founded} on perception of a token'.}  Only typically, though, because concrete strings can also be imagined and their types intuited in that way.  

But there are finite limits on the length and number of concrete inscriptions and even on the human capacity to imagine, so how do we get from here to a potentially infinite sequence of types? How do we come to know that `each string of strokes \textit{can} be extended by one more' \cite[\S 29, p. 173]{MR2381345}?  Parsons's answer involves imagining `an arbitrary string of strokes' (ibid.).  As with Locke's idea of a general triangle -- neither equilateral, nor isosceles, nor scalene -- 	

\begin{quote}
imagining an arbitrary string involves imagining a string of strokes without imagining its internal structure clearly enough to imagine a string of $n$ strokes for some particular $n$.  \cite[\S 29, p. 173]{MR2381345}
\end{quote}

\noindent He compares this with imagining a `crowd at a baseball game without imagining a crowd consisting of 34,793 spectators' (ibid., p. 173-174). This will form Parsons's basis:

\begin{quote}
	To see the \textit{possibility} of adding one more, it is only the general structure that we use, and not the specific fact that what we have before us was obtained by iterated additions of one more.  (Ibid., p. 175)
\end{quote}

Building on this idea of a general structure, Parsons offers two ways of justifying the claim that one more can always be added -- involving our perception of figure/ground or of temporal succession\footnote{\label{cog} The first is derived from our perceptual ability to `shift [our] attention so that what was previously ground is now figure' \cite[\S 29, p. 177]{MR2381345} and the other (developing Brouwer's `temporal two-ity') is derived from the fact that `we experience the world as temporal, and have the conviction that we can continue into the proximate future, in which the immediate past is retained' (ibid.).  Whatever Parsons may have intended, he offers no reason to suppose these are anything other than ordinary cognitive mechanisms studied in empirical psychology.}  -- with the upshot that `our insight into [the necessity that every string can be extended] is insight into experienced space and time' (ibid., p. 178).  The Kantian echo is obvious, but notice that the faculties involved in Parsons's Hilbertian intuitive picture\footnote{We call this an `intuitive picture' rather than a `concept' or `structure' to indicate that it's not an epistemic or metaphysical but a psychological pheonemonon.} (as we'll call it) aren't aspects of transcendental psychology or philosophical constructs like `rational intuition'\footnote{Parsons does discuss a form of rational intuition in Chapter 9 of \cite{MR2381345}, but in explicit contrast to the intuition of Chapter 5, the notion engaged here.  (Recall footnote \ref{cog}.)}   but aspects of ordinary human perception and cognition.

Perhaps it's unsurprising that a concept instantiated by such a structure, based in human imagination, would be considered potentially vague in one troublesome sense or other.  Like the imagined baseball crowd, our imaginings are indeterminate in indefinite range of aspects.  (Try imagining a perfect summer day or the Mona Lisa, then ask yourself how many questions you can honestly answer about the features of the resulting `image'.)  As for Parsons's imagined Hilbertian series of strings of strokes, `what if the picture began to flicker in the far distance?' (Wittgenstein, \textit{Remarks on the foundations of mathematics}, \cite[V.10, p. 268]{MR519516}).

What Parsons wants is assurance that nothing like this happens, that our concept of natural number specifies a unique and determinate structure.  Following Michael Dummett, he poses the problem in terms of communication:  how can I be sure that your numbers are like mine?\footnote{\label{dummett1} The only argument for this move – from a question about the determinateness of a structure to a question about communication – that we could find in the two Dummett essays that Parsons cites (\cite{MR157894,Dummett1967-DUMP-2}) appears in the course of his evaluation of a version of Platonism, in particular, one based on an analogy between intuition of abstract structures and ordinary perception.  If our intuition of the structure of natural numbers couldn't be communicated, `this would reduce the intuitive observation of abstract structures to something private and incommunicable – the analog not of observation of the physical world in the normal sense, but of the experience of sense data as conceived by those philosophers who hold these to be private and incommunicable' \cite[p. 210]{Dummett1967-DUMP-2}.  But Parsons's concern isn't with Platonism, much less this particular variety thereof; he's posed his challenge in terms of concepts – e.g., is your understanding of the concept of natural number the same as mine?  – not in terms of reference, as some do – are you talking about the same numbers as I am?  (See, e.g., the `objects-modelism' of \cite[pp. 144-145]{MR3821510}.)  This focus on concepts rather than ontology explains the focus on actual human faculties noted above.}  Maybe I'm working with one precification of our shared vague concept and you're working with another. Or, as Parsons prefers to pose the question, consider two speakers of a first-order language of arithmetic, Kurt and Michael, both of whom embrace the Peano axioms, and then, `we might ask how one could come to know that his ``numbers" are isomorphic to the other's' \cite[\S 49, p. 283]{MR2381345}.  This is the question Parsons hopes to answer by appealing to a categoricity argument. Exactly how a positive answer would bear on the original question of the determinateness of our concept is a topic for Subsection~\ref{subsec:par3}.\footnote{Recall footnote \ref{dummett1}.}

\subsection{What does he actually do (or can be done)?}\label{subsec:par2}

Given this goal, one's first thought is Dedekind's theorem (Subsection~\ref{subsec:ded2}), and that's precisely where Parsons's discussion begins.\footnote{See \cite[\S \S 10-11]{MR2381345}.} The first version of Dedekind's result, Theorem \ref{ded1}, is a weakly internal argument (the description of simply finite systems is used not mentioned\footnote{Recall footnote \ref{um}.}) based on what might as well be a weak first-order set theory.  Parsons would presumably agree to the conclusion Dedekind draws -- that the structure of the natural numbers can be described without appeal to temporal or geometric intuition -- though his  Hilbertian instance of that structure depends on both perceptual and imaginative intuitions.\footnote{Both Dedekind and Parsons appeal to roughly psychological instantiations of their structures -- Dedekind's thoughts of thoughts of thoughts and Parsons's Hilbertian perceived and imagined strokes -- to establish that their structural descriptions aren't vacuous. They differ when Dedekind insists on abstracting (creating) the numbers themselves; Parsons is content to rest with the structure.}   Still, to Parsons's mind, the role of set theory is problematic.  

Early on in \cite{MR2381345}, he writes that `we do not think of elementary number theory as involving commitment to properties, sets, classes, or Fregean concepts' \cite[\S 5, p. 20]{MR2381345}, a sentiment repeated in the later discussion of uniqueness:

\begin{quote}
	
	It is a commonplace in the foundations of mathematics that the idea of natural number is more elementary than that of set. \cite[\S 48, p. 274]{MR2381345}
\end{quote}

 \noindent Presumably, this means that Kurt could believe the Peano axioms while innocent of the concept of set, so this version of Dedekind's argument is of no use to him.\footnote{In response, it might be argued that the concept of finite set is essentially bound up with the concept of natural number, so that Kurt necessarily believes some set theory along with his assumptions about numbers.}  What Parsons wants, then, is a categoricity theorem for arithmetic that relies on no resources beyond what's available to Kurt himself, no resources beyond those of number theory itself -- what might be called a \textit{pure} categoricity argument.\footnote{We allude here to related themes of purity of method in the philosophy of mathematics. For an overview, see, e.g., \cite{Detlefsen2011-DETPOM}.}

The external version of Dedekind's result, Theorem \ref{ded2}, also depends on set theory and thus faces the same impurity objection that Parsons raises to Theorem \ref{ded1}.  Still, the introduction of second-order semantics reminds us that
 $$N \models \PATWO \land N' \models \PATWO \rightarrow N \cong N'$$ 
 \noindent is, in fact, a truth of second-order logic alone. In this form, it's obviously subject to a number-theoretic version of the challenge raised in Section~\ref{sec:kre}, the one Kreisel hasn't fully overcome:  if you're worried about the determinateness of the range of predicates allowed in the open-schematic induction axiom, as Parsons is, then you're also nervous about the determinateness of second-order quantification.  More in line with Parsons's thinking,\footnote{See \cite{MR2381345}, pp. 45-46.} the second-order logical truth can be recast without semantics, as a weakly internal theorem in syntactic second-order logic. Let $\PATWO(X,Y,Z)$ be the second-order Peano axioms formulated for a unary predicate, a unary function symbol, and a constant, and let $\ISO((X,Y,Z), (X',Y',Z'))$ be 
 the second-order sentence claiming that there's an isomorphism between $(X,Y,Z)$ and $(X',Y',Z')$. Then\footnote{\label{BWonP}$\vdash_2$ is deriviability in the usual axiom system of second-order logic:  first-order axioms, introduction and elimination rules for the second-order variables, plus all instances of comprehension and choice (see e.g., \cite[pp. 66-67]{MR1143781}, \cite[p. 34]{MR3821510}.)) Button and Walsh, ibid., pp. 240-241, recommend that Parsons content himself with this result.  The text indicates why Parsons rules this out.}

\begin{theorem}\label{dedpure} 
$\vdash_2 \forall N_1, S_1, 0_1, N_2, S_2, 0_2((\PATWO(N_1, S_1, 0_1) \land \PATWO(N_2, S_2, 0_2))\rightarrow \ISO((N_1, S_1, 0_1), (N_2, S_2, 0_2)))$.
	\end{theorem}

\noindent The proof involves only $\Pi_1^1$-comprehension, but Parsons echoes the same concern about even that much:  there's `a question whether the second-order quantifiers ... can have a definite sense' (\cite{MR2381345}, \S 47, p. 270). In particular, even $\Pi_1^1$-comprehension is committed to a well-determined range of predicates of natural numbers.\footnote{Presumably he'd raise an analogous objection to the  set-quantifiers in the weak set-theoretic versions of Dedekind's argument, even apart from the impurity problem.} 

So, given Kurt's epistemic situation, Parsons wants a pure theorem.  Given his objections to second-order logic, either semantic or syntactic, it must also be first-order.  Finally, in light of these two desiderata, any appeal to semantic notions is problematic, so such an argument must be weakly internal.  (Notice that the  preference for internalness isn't for its own sake but as a means toward achieving the other two desiderata.)  To sum up:  the goal is a pure weakly internal argument -- let's call this \textit{strongly internal} or just \textit{internal}\footnote{For the record, the term `internal categoricity' was first used in \cite{Walmsley2002-WALCAI-3}.  After \cite{10.2307/23350653}, the idea occurred more or less explicitly throughout the 1990s and beyond, for example, in \cite{MR1143781}, \cite{MR1304680}, \cite{McGee1997-MCGHWL}, and \cite{Lavine1999}, as well as  \cite{MR3326591}   and \cite{theo.12237}. }  -- that's also first-order. This is the bar Parsons sets himself.

Before asking how Parsons undertakes to do this and how his sketch can be developed, we pause to note some fundamental differences between Dedekind's project and Parsons's -- in the role of a background metaphysics and in motivation.  Both begin with a kind of pre-theoretic metaphysics:  Dedekind his sets and Parsons his structures.  (We haven't and won't attempt to pinpoint their precise ontological status – realistic, idealistic, or otherwise.) Dedekind begins by codifying a few relevant features of his sets and using those assumptions to study something else (the concept of natural number).  In contrast, Parsons never produces an explicit theory of structures -- they remain pre-theoretic, doing no explicit mathematical work -- but they are, nevertheless, the ultimate object of study (in particular, the natural number structure).  On motivation, Dedekind has no doubts about the determinateness of his concept of natural number, he's only interested in what it takes to characterize the relevant structure, while for  Parsons, the whole inquiry arises from his worry that the concept might be vague or otherwise indeterminate.  So, despite the fact that both are focused on the concept of natural number, despite the fact that Parsons appropriates the mathematical core of Dedekind's theorem, their philosophical perspectives and their goals are quite different.

So, how does Parsons hope to pull this off; how is Kurt to prove that his numbers are isomorphic to Michael's using only first-order number-theoretic tools?  We know the first-order Peano axioms aren't enough, but we've ruled out a move to second-order.  Perhaps ironically, Parsons hopes to locate a happy medium between too weak and too strong by exploiting the open-schematic character of the induction rule in his characterization of the concept of natural number, the very feature that led to his worry over vagueness in the first place. 

 The idea is that open-schematic induction takes us beyond the usual first-order axiom, which allows only formulas from a fixed vocabulary, but not far enough to generate a second-order quantifier:

\begin{quote}
	In … referring to arbitrary predicates, the statement of the rule makes no assumption about what counts as a predicate. …  The rule is not a generalization over a given domain of entities and could not be, because it is not determined what predicates will or can be constructed and understood.  … Replacing the rule by an axiom, that is, a single statement expressed by quantification over sets … is not suitable for the explanation of the number concept by rules that we have been engaged in.  (\cite[\S 47, pp. 269-270]{MR2381345})
	
	This understanding of induction implies that the applicability of the rule is not limited to predicates defined in some particular first-order language such as that of first-order arithmetic.  But we must not take it as implying the unavoidability or even the legitimacy of for second-order logic. … second-order logic is not forced on us.  \cite[\S 47, p. 270]{MR2381345}
\end{quote}
Not all observers agree on this last point.  For example, see Shapiro  – `the effect is the same as allowing only initial universal quantifiers over second-order variables' \cite[p. 246]{MR1143781} – or Field  – `so we have an analog, in the schematic first-order case, of assertions of $\Pi_1^1$ second-order sentences' \cite[p. 354]{field1} – or Button and Walsh  – `this is to fall back on (a syntactic fragment of) full second-order logic … our grasp on the idea of \textit{totally} open-ended induction is exactly as precarious as our grasp on \textit{full} second-order quantification' \cite[pp. 240, 163]{MR3821510}.

This is a debate we hope not to engage.  If these observers are wrong and open-schematic induction isn't just a re-description of a weak form of second-order logic, then it's hard to see how a precise mathematical treatment could be given; it would require a formal account of `any well-defined predicate' that avoids `specifying exactly what the range of such predicates is' \cite[\S 48, pp. 272-273]{MR2381345}.  We prefer a different route in pursuit of Parsons's goal:  investigating what unproblematically first-order resources Kurt would need to argue for the desired isomorphism.

So, suppose Kurt speaks of his numbers, his successor function, and his first number, Michael speaks of his, and they believe many of the same arithmetic truths. How can Kurt be sure they're talking about the same structure?  To pose the question more precisely, let $N_1$ and $N_2$ be unary predicate symbols, $S_1$ and $S_2$ unary function symbols, and $0_1$ and $0_2$ constant symbols.\footnote{\label{fn}As it happens, $+$ and $\cdot$ are definable from $0$ and $S$ in second-order $\PA$, but not in first-order $\PA$, so it would make more sense to include them here. We remain faithful to Parsons's presentation for now, but switch to including $+$ and $\cdot$ below, when the question is how to replicate something like second-order internal categoricity in a first-order setting.} Then let $T_1$ be the first order Peano axioms for $(N_1,0_1,S_1)$ with induction formulated as a schema in the vocabulary $\{N_1,0_1,S_1\}\cup \{N_2,0_2,S_2\}$, that is,

 \begin{equation}\label{schema}
	\begin{array}{l}
		(\psi(0_1,\vec{y})\wedge\forall x((N_1(x)\wedge\psi(x,\vec{y}))\to
		\psi(S_1(x),\vec{y})))\to\\
		\qquad\forall x(N_1(x)\to\psi(x,\vec{y})),
	\end{array}
\end{equation}
where $\psi(x,\vec{y})$ is any first-order formula in the vocabulary $\{N_1,0_1,S_1\} \cup \{N_2,0_2,S_2\}$. Let $T_2$ be the corresponding theory for ($N_2,0_2,S_2$).

 Kurt's native language is built on the vocabulary $\{N_1, S_1, 0_1\}$, but his theory of natural numbers is $T_1$, which extends his schematic induction principle to allow vocabulary from Michael's language, built on $\{N_2, S_2, 0_2\}$.  This means that the formula $\psi$ in the induction schema of $T_1$ can include vocabulary beyond the native language of ($N_1,0_1,S_1$) into the language of ($N_2, S_2, 0_2$) and that its quantifiers range beyond $N_1$ to include $N_2$.   This is stronger than what we ordinarily mean by first-order Peano arithmetic; by including more instances of induction, it moves in the direction of second-order $\PA$.  (Obviously, this is also true of Michael, mutatis mutandis.) In addition, Kurt's understanding of Michael entails that he recognizes $T_2$ as holding in that added ontology, so his working theory includes $T_1 \cup T_2$.  
 
 With this added strength, Kurt, arguing in $T_1$, can use his induction schema to prove first-order statements about relationships between the two vocabularies. This raises the possibility that Kurt's $T_1 \cup T_2$ might be enough to produce the isomorphism he's hoping for, an isomorphism $\phi$ between ($N_1, S_1, 0_1$) and ($N_2, S_2, 0_2$).\footnote{In embryonic form, this idea is due to  \cite{MR626358}.}   So, is $T_1 \cup T_2$ enough to prove categoricity?As it happens, the answer is no.\footnote{Button and Walsh make a similar observation in their Proposition 10.4 (\cite{MR3821510}, \S 10.C).}

\begin{remark}
	$T_1 \cup T_2$ cannot produce such a $\phi$.
\end{remark}

\begin{proof}
	Our goal is to construct a  model $M'$  of $T_1\cup T_2$ such that structures $(N_1^{M'},0_1^{M'},S_1^{M'})$ and $(N_2^{M'},0_2^{M'},S_2^{M'})$ are non-isomorphic -- so that Parsons's desired $\phi$ cannot exist.  First, let $M$ be a model of the vocabulary of $T_1\cup T_2$ such  that $M$ is the disjoint sum of two copies of the standard model $(\oN,0,S)$. Thus $(N_1^M,0_1^M,S_1^M)\cong(\oN,0,S)$ 
	and $(N_2^M,0_2^M,S_2^M)\cong(\oN,0,S)$. Then let $(\oN^*,0^*,S^*)$ be a countable non-standard model elementarily equivalent to (but again disjoint from)  $(\oN,0,S)$ and let $M'$ be the disjoint sum of $(\oN,0,S)$ and $(\oN^*,0^*,S^*)$, making it a model of the vocabulary of $T_1\cup T_2$. A simple Ehrenfeucht-Fra\"\i ss\'e-game argument, as in  \cite{MR0360196}, shows that $M\equiv M'$. Since $(\oN,0,S)$ satisfies even the second-order Induction Axiom, $M$ certainly satisfies $T_1$. Respectively, $M$ satisfies $T_2$. Thus $M\models T_1\cup T_2$ and then also $M'\models T_1\cup T_2$. But  $(N_1^{M'},0_1^{M'},S_1^{M'})$ and $(N_2^{M'},0_2^{M'},S_2^{M'})$ are non-isomorphic, as the former is standard and the latter is non-standard. 
\end{proof}

Though this argument demonstrates that $T_1 \cup T_2$ isn't enough to generate the desired $\phi$, it also gives a hint of what more is needed. The two non-isomorphic countable models in the proof are $n$-equivalent for all $n$. For countable structures, isomorphism is equivalent to the existence of a winning strategy for the so-called isomorphism player in the Ehrenfeucht-Fra\"{i}ss\'{e} game of length $\omega$. So the point of the above counter-example is that the models imitate being isomorphic, in the sense that the isomorphism player has a winning strategy, up to $n$ moves of the Ehrenfeucht-Fra\"{i}ss\'{e} game, for every $n$ separately, but if infinitely many moves are allowed, the non-isomorphism player has a winning strategy.  The finite strategies corresponding to different finite lengths of the  game do not cohere. We need a global principle that knits the different finite pieces together, essentially a winning strategy in the infinite game. The strength gained by replacing the usual induction schemas in two copies of first-order $\PA$ with the more generous induction principles of $T_1$ and $T_2$ isn't enough, but there are enhancements that are up to the task.

Parsons's preferred enhancement is primitive recursion.\footnote{See \cite[\S 49, p. 281]{MR2381345}: `in keeping with Skolem's recursive arithmetic, we can introduce by primitive recursion a functor'.}  In particular, he introduces his  $\phi$ by recursion on $N_1$:
$$\left\{\begin{array}{l}
	\phi(0_1)=0_2\\
	\phi(S_1(x))=S_2(\phi(x)).
\end{array}\right.$$
He then reprises Dedekind's proof from Subsection~\ref{subsec:ded2};  to show $\phi$ is one-to-one, he applies $N_1$-induction to $\forall m (\phi(n)=\phi(m) \rightarrow n=m)$; to show $\phi$ is onto $N_2$, he applies $N_2$-induction to $\exists m (\phi(m)=n)$.  (He notes that both instances of induction are first-order.)  Clearly this $\phi$ would be the required isomorphism.  

Unfortunately, Parsons's brief presentation doesn't extend beyond his gesture toward `Skolem's recursive arithmetic', so the challenge is to formulate a first-order internal theorem that captures the required recursion. The standard method involves some version of $\Pi^1_1$-comprehension and a second-order induction axiom -- and, as we know, Parsons's vision requires that second-order logic, however weak, be abandoned altogether.  What he wants is a first-order extension of $T_1 \cup T_2$ that's capable of generating the required isomorphism between Kurt's numbers and Michael's. Precisely what it takes to do this is a difficult technical question because of the way  $\phi$ crosses from one vocabulary to another -- existing results in reverse mathematics deal with one version of arithmetic at a time.  We explore what's known of the mathematical situation, drawing on and extending the second author's work in \cite{MR3326591} (jointly with Wang)  and \cite{theo.12237}. First-order formulations introduce some non-trivial complications, so we  begin with the second-order case as a warm up. 
 
Consider, then, a second-order theory $T_1^2$ in the vocabulary $\{R_1,\ldots,R_n\}$, where  $R_i$ is $r_i$-ary for $1\le i\le n$.\footnote{We assume for simplicity that the vocabulary is relational and finite. Neither constraint is essential. When we refer to constants, such as $0$ in $\PATWO$, we think of them here as unary relations which are singletons.}  Suppose $\{R'_1,\ldots,R'_n\}$ are new relation symbols  such that the arity of $R'_i$ is $r_i$ for $1\le i\le n$, and $T_2^2$ is the theory $T_1^2$ with every occurrence of $R_i$ replaced by $R'_i$, $1\le i\le n$. Let $F$ be a new unary function symbol, $U$ and $U'$ be new unary predicate symbols, and $\Iso(F,U,U')$  the first-order formula which says that $F$ is a bijection between $U$ and $U'$ such that $R_i(a_1,\ldots,a_{r_i})\leftrightarrow R'_i(F(a_1),\ldots,F(a_{r_i}))$ for all $a_1,\ldots,a_n\in U$ and all $i$ with $1\le i\le n$. If $\phi$ is any second-order formula, we let $\phi^{(U)}$ denote the formula obtained from $\phi$ by relativizing all first- and second-order quantifiers to $U$. Finally, for any theory $T$, $T^{(U)}$ denotes  $\{\phi^{(U)}:\phi\in T\}$.  We then formally define:
 
\begin{definition}\label{definition1}The second order theory $T_1^2$
is \emph{internally categorical} if 
\begin{equation}\label{theta}
T_1^{2(U)}\cup T_2^{2(U')}\vdash_2 \exists F\hspace{2pt}\Iso(F,U,U').
\end{equation}

\end{definition}

\noindent $T^{2(U)}_1\cup T^{2(U)}_2$ appears on the left side of the turnstile -- as opposed to the right side, which may seem more natural for categoricity claims -- simply to allow for infinite theories.

Intuitively, $T_1^2$ stipulates some axioms for a relational structure with   the relations $\{R_1,\ldots,R_n\}$, and it does so without semantic notions. Similarly, $T_2^2$ stipulates the same axioms for the relational structure with   $\{R'_1,\ldots,R'_n\}$ as the relations. The sentence $\Iso(F,U,U')$ says that $F$ is an isomorphism between the two relational structures when their domains are restricted to $U$ and $U'$, respectively. The categoricity of $T_1^2$ is obviously equivalent to the statement that  the second-order sentence $\exists F\hspace{2pt}\Iso(F,U,U')$ is true in any full model of $T^{2(U)}_1\cup T^{2(U)}_2$. But (\ref{theta}) says that the categoricity of $T_1^2$ is not just a set-theoretical fact but is, in fact, provable in second-order logic. Semantically, this means that every Henkin model of $T^{2(U)}_1\cup T^{2(U')}_2$ satisfies $\exists F\hspace{2pt}\Iso(F,U,U')$. In other words, if a Henkin model recognizes two models of $T_1^2$, it also recognizes an isomorphism between those two models.

Notice, then, that internal categoricity is more general than our familiar notion of categoricity:

\begin{theorem}
Suppose $T_1^2$ is internally categorical.  Then $T_1^2$ is categorical, i.e. if $M \models T_1^2$ and $M' \models T_1^2$, where $M$ and $M'$ are full models, then $M\cong M'$.

\end{theorem}

\begin{proof}W.l.o.g., suppose the domains of $M$ and $M'$, which we denote also with $M$ and $M'$, are disjoint. Let $M^*$ be the unique model with $M\cup M'$ as  domain, $U^{M^*}=M$, ${U'}^{M^*}=M'$,  ${R}_i^{M^*}=R_i^{M}$, and ${R'}_i^{M^*}={R}_i^{M'}$ for $1\le i\le n$. Clearly, $M^*\models T_1^{2(U)}\cup T_2^{2(U')}$. Hence by internal categoricity, 
$$(U^{M^*},R^{M^*}_1,\ldots,R^{M^*}_n)\cong({U'}^{M^*},{R'}^{M^*}_1,\ldots,{R'}^{M^*}_n).$$
Now $M\cong M'$ follows. \end{proof}

\noindent In particular, when $T_1^2$ is $\PATWO$, we have

\begin{theorem}[e.g. \cite{MR1029277}]\label{folklore} $\PATWO$ is internally categorical.
\end{theorem}

\noindent Dedekind's original result follows immediately:

\begin{corollary}[Dedekind]
$\PATWO$ is categorical.
\end{corollary}

\noindent This brings the story of internal categoricity for second-order theories full circle.

Despite the  purity of Theorem~\ref{folklore}, its second-order character obviously rules it out for Parsons's purposes, so let's now see to what extent this general approach can be adapted to a first-order context.  For various reasons that aren't directly relevant here,\footnote{For theories with  first-order schemas arising from  second-order $\Pi^1_1$-axioms, it's clear how to define internal categoricity, but for first-order theories in general the question is open.}  it's not easy to define first-order internal categoricity for general theories $T$.  Instead, since our particular interest is in theories of arithmetic, we explore possibilities for a first-order analog to Definition~\ref{definition1} that's designed specifically for the case of the Peano axioms.  To that end, suppose  the vocabulary of $\PA$ is $\{+,\cdot,0,1\}$.\footnote{Recall footnote~\ref{fn}.}  Let $N_1$ and $N_2$ be unary predicate symbols,  $+_1$ and $+_2$, $\cdot_1$ and $\cdot_2$ binary function symbols, and $0_1$ and $0_2$, $1_1$ and $1_2$ constant symbols.  Then let  $\PA_1(N_1)$ be the first-order theory $\PA$ written in the vocabulary $\{+_1,\cdot_1,0_1,1_1\}$, with the functions $+_1,\cdot_1$ mapping $N_1\times N_1$ to $N_1$, the constants $0_1,1_1$ in $N_1$, but with adjustments reminiscent of $T_1$:  the Induction Schema allows formulas from the larger vocabulary $\{+,\cdot,0,1\} \cup \{+_1,\cdot_1,0_1,1_1\}
\cup \{+_2,\cdot_2,0_2,1_2\}$ and (first-order) quantifiers range over the whole domain, including $N_1 \cup N_2$. Likewise $\PA_2(N_2)$.  

The first, obvious obstacle to a first-order project  is that the basic claim of categoricity -- `there's an isomorphism between $N_1$ and $N_2$'-- can't be formulated in our first-order language.  It might be possible, however, to devise a first-order formula, $\phi(x,y)$, that determines a functional relation of the desired form.  Devising an appropriate $\phi$ involves coding.  Let $\psi(x,u,v)$ say that $x$, using $+$ and $\cdot$, codes an initial segment $I_1$ of $N_1$ ending with $u$, an initial segment $I_2$ of $N_2$ ending with $v$, and a function $f:I_1 \rightarrow I_2$ such that $f(0_1)=0_2$, $f(z+_11_1) = f(z)+_21_2$ for all $z \in I_1$ preceding $u$ in $I_1$, and $f(u)=v$.  Then let $\phi (u,v)$ be $\exists x \psi(x,u,v)$.  Saying that $\phi$ is an isomorphism is straightforward:  let  $\ISO_\phi(N_1,N_2)$ be the first-order formula which says that $\phi$ is a bijection between $N_1$ and $N_2$, and for all $x,y\in N_1$:
\begin{equation}\label{F}
\begin{array}{l}
F(0_1)=0_2\wedge\\
F(1_1)=1_2\wedge\\
F(x+_1y)=F(x)+_2 F(y)\wedge\\
 F(x\cdot_1y)=F(x)\cdot_2 F(y), 
\end{array}
\end{equation}

\noindent where $F(x)$ abbreviates the unique $y \in N_2$ such that $\phi(x,y)$. 

With this machinery in place, a version of categoricity can be proved from a collection of first-order arithmetic assumptions.  With $\phi$ as above:

\begin{theorem}[\cite{theo.12237}]\label{theorem2}First-order $\PA$ is
 \emph{internally categorical} in the sense that
\begin{equation}\label{theta1}
\PA\cup\PA_1(N_1)\cup \PA_2(N_2)\vdash \ISO_\phi(N_1,N_2).
\end{equation}
\end{theorem}

\begin{proof}
By induction on $+_1$, we show that $F$ is a function $N_1\to N_2$ and  that it  satisfies the conditions (\ref{F}). We then use induction on $+_2$ to prove that $F$ is onto.  These proofs exploit the fact that induction holds for first-order formulas that have any of the symbols $+,\cdot,0,1,+_1,\cdot_1,0_1,1_1,+_2,\cdot_2,0_2,1_2$. Induction on $+$ is used to establish the necessary properties of the coding.
\end{proof}

Intuitively speaking, from the perspective of $\PA$, the rest of the left-hand side of the turnstyle in (\ref{theta1}) tells us that there are two models $(N_1,+_1,\cdot_1,0_1,1_1)$ and $(N_2,+_2,\cdot_2,0_2,1_2)$ with $N_1, N_2 \subseteq N$, both of which satisfy the axioms of first-order $\PA$. From these assumptions, Theorem~\ref{theorem2} tells us that $\phi$ is an isomorphism between them.  As with second-order internal categoricity (Theorem~\ref{folklore}), this is true even if the model of $\PA$ we're working in is non-standard or uncountable. It's easy to see that the internal categoricity (and thereby the categoricity) of $\PATWO$ follows from the internal categoricity of the first-order $\PA$ in the sense of Theorem~\ref{theorem2}.

Though it comes close, the sense of `internal categoricity' in Theorem~\ref{theorem2} doesn't look quite like that of Definition~\ref{definition1}, largely because of that extra $\PA$ on the left-hand side. But note that, in Definition~\ref{definition1}, the comprehension axioms are built into $\vdash_2$.  This is where the hidden power of second-order comprehension comes into dramatic focus:  it's simply assumed that the range of the second-order quantifier includes an item for every formula in the combined vocabularies of $T^{2(U)}_1\cup T^{2(U)}_2$; this is what generates the required bijection, the required link between the two models.  In the first-order case, where no such implicit background is in place, the extra PA is required to play this role.  (And, for the record, Dedekind's weak set theory does the same in his Theorem \ref{ded1}.) 
 
 Another option is to consider two copies of $\PA$ with the same domain:  $\PA_1^*$ says that the axioms of first-order $\PA$ hold of $+_1,\cdot_1,0_1,1_1$, $\PA_2^*$ says the same for  $+_2,\cdot_2,0_2,1_2$, where both allow induction for first-order formulas in the joint vocabulary $\{+_1,\cdot_1,0_1,1_1\}
\cup \{+_2,\cdot_2,0_2,1_2\}$.  If $\AUT_\phi$  says that $\phi$ defines a permutation of the domain and that it satisfies the preservation clauses in the definition of $\ISO_\phi(N_1,N_2)$, then we get this result:

\begin{theorem}[\cite{MR3984974,theo.12237}]\label{theorem3}First-order Peano arithmetic $\PA$ is
 \emph{internally categorical} in the sense that
\begin{equation}\label{theta2}
\PA^*_1\cup \PA^*_2\vdash \AUT_\phi.
\end{equation}
\end{theorem}

\noindent This mimics Definition~\ref{definition1} more closely:  there's no extra $\PA$ overseeing the two copies $\PA^*_1$ and $\PA^*_2$.  Still, the assumption that the two copies have the same domain (which, incidentally, needn't be countable) makes Theorem~\ref{theorem3} less compelling than Theorem~\ref{theorem2}. Intuitively speaking, we're trying to establish that two models of first-order arithmetic are isomorphic, and (knowing that this is impossible without some extra assumption) we make the task easier by assuming that we are `half way there' in the sense that at least the two models have the same domain.

\subsection{Does what he did (or can be done) accomplish what he set out to do?}\label{subsec:par3}

Does any of this help Kurt determine whether or not Michael's numbers are isomorphic to his own?  In Theorem \ref{theorem2}, he begins from $\PA$ -- not $\PA_1(N_1)$, his own idiosyncratic version --  and views both $\PA_1(N_1)$ and Michael's version from that broader point of view.   It's hard to see how this could honestly be described as Kurt proving, exclusively from his own perspective, that his numbers are isomorphic to Michael's.\footnote{There's also the murkier question of whether Kurt -- whose thinking, we assume, is relentlessly first-order -- could replicate the line of thought that came before Theorem~\ref{theorem2}.  It might be argued that understanding the significance of $\phi$ depends on seeing it as a first-order replacement for the second-order $\exists f$, which would be reason to doubt that a relentlessly first-order Kurt would see its significance, let alone be inspired to think it up.  (Cf. \cite[p. 148]{Kreisel1967-KREIRA}.)  For that matter, how would a relentlessly first-order Kurt pose the question of categoricity to himself in the first place? } Theorem~\ref{theorem3} would return Kurt to his own perspective, but he would need to know in advance that Michael is talking about the same things he is -- though perhaps getting some of their properties and relations wrong -- and it's hard to see what grounds he could have for this belief.\footnote{Obviously, the concern of the previous footnote arises here as well.} So far, at least, the mathematics of internal categoricity hasn't delivered a first-order result that Kurt can deploy, hasn't shown that Parsons can avoid an appeal to second-order logic. This obviously isn't conclusive, but the challenge hasn't yet been met.

On the philosophical aspects of Parsons's effort, consider again the question he set out to answer:  is our shared concept of natural number univocal, or can `two [different] structures answer equally well to our conception of the sequence of natural numbers' \cite[\S 48, p. 272]{MR2381345}?  As we've seen, he then characterizes our shared concept in terms of a set of rules that include the open-schematic version of mathematical induction and offers his  Hilbertian intuitive picture of ever-increasing strings of strokes as an example of a structure that answers to it. As noted earlier, Parsons's account of this picture calls on ordinary human psychology -- perception of figure/ground or temporal succession, `insight into experienced space and time' (ibid., \S 29, p. 178) -- so the number concept it answers to would seem to be one shared by  adult humans beyond the developmental levels of understanding counting (about age 3) and realizing there's no largest number (about age 8).  It's the shared understanding of 1, 2, 3, ...  that we use every day, in contexts both mudane and theoretical.

Now consider some recent psychological findings of Josephine Relaford-Doyle and Rafael N\'{u}\~{n}ez at UCSD (\cite{rd2017,MR3930037,rd2021}).  The experimental subjects were college students, one group drawn from the general population without training in mathematics and another from mathematics majors with at least a B- in a course on proof techniques that included mathematical induction. Subjects were shown a visual argument that the sum of the first $n$ odd numbers is $n^2$:  represent the first odd number, $1$, by one dot on the page.  ($1=1^2$).   Represent the second odd number, $3$, by adding three dots so as to form a $2 \times 2$ square with the original dot in the lower, left-hand corner.  ($1+3=2^2$)  Represent $5$ by adding five dots in the same way so as to form a $3 \times 3$ square, again with the first dot in the lower left-hand corner. ($1+3+5=5^2$).  And so on.  The recursive character of the construction is clear; an induction to all natural numbers is implicit.

Or so it seems.  The experimenters focused on subjects who demonstrated that they understood the statement (the sum of the first $n$ natural numbers is $n^2$) and were willing to generalize to nearby numbers (having gone through the argument for the first five odd numbers, they believed it would hold for seven or eight).  These subjects were then asked a striking question:  might there be a large number for which the statement fails?  The answer seems obvious – surely not! – but that was the reaction of only a third of the mathematically trained respondents and fewer than 10\% of the rest; nearly 40\% of the untrained subjects and 10\% of the trained respondents saw the existence of such a counterexample as plausible.  The thought was that `larger numbers could be … outliers' or `when it gets to really high numbers … it's possible that … maybe it gets kind of fuzzy.  Because at extremes things tend not to work as they do normally' or perhaps most tellingly, `it's too hard to draw a million dots' \cite[p. 14]{rd2021}. One subject, a math and computer science major in the mathematically trained group, made explicit reference to computational limitations: `Large numbers have problems.  In computer science … even if it works on small, medium, most numbers, when you get to the large number, there's not enough space and it doesn't work anymore.  So large numbers are weird cases'  (\cite[p.  248]{MR3930037}).

 Relaford-Doyle and N\'{u}\~{n}ez summarize:
 
 \begin{quote}
 	These responses all suggest that these participants believe that large numbers may have qualitatively different properties than small numbers, such that rules that apply to small numbers may no longer work at larger magnitudes.  This is a reasonable conclusion to draw – in practice there are many differences between small and large numbers: small numbers … are encountered more frequently, have simple numerical notation and lexical structure, and are easier to use in computations.  However, this finding is surprising in that it is in opposition to the widely-held assumption in developmental psychology that `mature' conceptualizations of natural number are consistent with the Dedekind-Peano axioms, in which the entire set of natural number is governed by the same logic.  \cite[p. 14]{rd2021}
 \end{quote}
 
\noindent They  conclude that `many undergraduates may possess non-normative conceptualizations of natural number' \cite[p. 18]{rd2021}.

The subjects' remarks suggest that the `reasonable conclusion' they've drawn is something like:  the diagrammatic reasoning demonstrates the truth of the target claim only for numbers to which it could be applied, numbers for which there could be an array of dots, numbers up to which one could count -- what we might informally call `feasible' numbers. This idea generates a sorites paradox -- 1 is feasible, if $n$ is feasible, so is $n+1$, but $10^{70}$ is not -- and nevertheless plays a leading role in Yessenin-Volpin's strict finitist program in the foundations of mathematics (\cite{MR0295876}) and appears in computer science as `feasible computation' (see, e.g., \cite{WDcom}).  These interconnections have been investigated philosophically by  \cite{Dummett1975-DUMWP} and more recently and formally by Walter \cite{MR3834011}, but our concern here is with the informal notion, the intuitive picture that appears to lie behind the reactions of Relaford-Doyle and N\'{u}\~{n}ez's subjects.  Dean demonstrates the consistency of this picture with a nonstandard model of arithmetic in which an `infinite integer' plays  the role of an unfeasible number.  This isn't to say that these students -- let's call them `feasibilists' -- are referring to a non-standard model.  Theirs is an alternative intuitive picture on a psychological level analogous to that of Parsons's Hilbertian notion,\footnote{Again, we're assuming that Parsons's Hilbertian intuitive picture is straightforwardly psychological, with no transcendental element.  (Recall footnote \ref{cog}.)  That's why we take empirical psychology to be relevant.} not a model in the sense of model theory, and that alternative intuition prompts them to disagree with those hewing to the orthodox concept.

If these preliminary results are replicated and extended, one surprise would be that the developmental story Relaford-Doyle and N\'{u}\~{n}ez allude needs an update:  believing there's no largest number, appreciating the endlessness of the sequence, 1, 2, 3 … , isn't always enough to equip one with the notion of an orthodox omega sequence.\footnote{The first author made something like this 
mistake, e.g.,  in \cite[\S 2]{MR3930026}, though she left open the possibility that our shared notion isn't fully determinate.  Thanks to Jessica Gonzales, Christopher Mitsch, and  Stella Moon (attendees at a seminar offered by N\'{u}\~{n}ez at UCSD)  for calling the work of Relaford-Doyle and N\'{u}\~{n}ez to her attention.},\footnote{This would be particularly interesting developmentally, because it would locate a third and quite different linguistic, or perhaps better social, contribution to the mathematical conception of the natural numbers.  In counting, the initially merely sing-song sequence of number words ends up providing a bridge between two primitive cognitive systems that we share with many other animals (one exact system for small numbers and one approximate system for larger numbers).  After that, coming to realize there's no largest number depends on increasing dexterity with producing ever larger numerical expressions.  This new and final third step would typically involve inculcation with mathematical induction.}  For Parsons's project, the surprise appears to be just what he feared:  that our shared concept of natural number is not univocal; that it admits both feasibilist and more orthodox readings; that it doesn't include full  mathematical induction in the usual sense.  (The feasibilist apparently does think proof by induction works for feasible portion of the number sequence.)  Parsons's emphasis on the centrality of induction suggests that he might react by denying the feasibilist has a concept of natural number; he might restrict his circle of concept-sharing to those who've been trained on and internalized full mathematical induction (or perhaps come by it intuitively).\footnote{Relaford-Doyle and N\'{u}\~{n}ez \cite[p. 1]{rd2021} begins with the observation that `Studies have repeatedly documented students' considerable difficulties in learning mathematical induction'.  Along the way, they note that even some students trained in mathematical induction understand it only at the `procedural' level – that is, as a sort of algebraic recipe – and not at the `conceptual' level that would allow them to recognize induction in the diagrammatic argument.}

This path seems uninviting. The subjects in these experiments are well-educated young people who operate successfully with numbers in their everyday and intellectual lives, so it doesn't seem plausible to claim that they lack the concept of natural number.  Of course, Parsons is within his rights to limit the scope of his inquiry and to use the term `concept of natural number' as he chooses, but his Hilbertian intuitive picture is intended to answer to that concept, and the psychological description of the intuitive picture is clearly rests on our most fundamental thinking about numbers.  The feasibilists among Relaford-Doyle and N\'{u}\~{n}ez's subjects surely perceive figure/ground relations and temporal succession, surely experience space and time, so they must share the picture of the endless series of strings of strokes -- $\|$, $\||$, $\|||$, ... -- with their mathematical betters.  Rather than painting the feasibilist as embracing an alternative to the Hilbertian picture, perhaps a better way to describe the situation is to say that we do all share that picture, but that the picture itself is vague in Parsons's sense, that it can be taken in (at least) two different ways, that it does, in a sense, `flicker in the distance'.  The threat of this kind of psychological divergence isn't the sort of thing susceptible to the proving of theorems, about categoricity or anything else. But the empirical situation remains unsettled.  

To sum up, if this general psychological question about our shared concept of number is what Parsons ultimately cares about, and if the experimental results hold, then his question -- is it univocal? -- would be answered in the negative.  If the experimental results don't hold or if his interest is actually in the narrower question of the more learned concept of number, we're left with the mathematical question of how Kurt can define his isomorphism within the bounds of a first-order internal categoricity theorem -- and the closest known  mathematical results, Theorem~\ref{theorem2} and Theorem~\ref{theorem3}, fall short of that goal.  In the end, there seems room for doubt that our shared concept, Parsons's own Hilbertian intuition of the endless sequence of strokes, is as clear and determinate as we think it is.  And if there is this room for doubt, formal categoricity theorems don't seem to be the kind of thing that might conceivably help.   Given these open questions, both mathematical and philosophical, Parsons's appeal to categoricity arguments to establish `the uniqueness of the natural numbers' can't yet be judged a success.

\section{Button and Walsh in `Categoricity' (2018)}\label{sec:but}

\subsection{What do they set out to accomplish?} \label{subsec:but1} 

As Kreisel points out (Section~\ref{sec:kre}), anyone confident that the quantifiers of second-order logic are determinate should take the second-order version of Zermelo's categoricity theorem, Theorem \ref{ZermeloCatFull}, to show that CH has a determinate truth value, and in fact, various observers make exactly this  inference.\footnote{See, e.g., Shapiro \cite[p. 105]{MR2993116}, Hellman \cite[pp. 70-71]{MR1029277}.}  Timothy Button and Sean Walsh in their 2018 book \textit{Philosophy and Model theory} (\cite{MR3821510}) are not among them.  Part B of that book, `Categoricity', includes an extended exploration of the possible bearing of categoricity on the status of CH, with analyses of both external and internal theorems. For most of the book, they adopt a philosophical perspective they call `modelism', but in Part B, they introduce an alternative, `internalism'.  Their official position is to endorse neither --

\begin{quote}
	We should emphasise right now  … that we are not \textit{advocating} internalism, any more than we were advocating modelism.  Rather, we are presenting internalism in a speculative spirit.  It is a fascinating position, worthy of attention, and we want to develop it as best we can.  (\cite{MR3821510}, p. 223)
\end{quote}

\noindent-- so our goal here is to examine their claims about what would follow if one were to adopt one perspective or the other.

Modelism, then, is a philosophical perspective on the nature of model theory, a branch of mathematics that takes place within set theory.\footnote{See, e.g., \cite{MR3821510}, p. 145:  `Model theory is officially implemented within set theory'.}  The modelist understands model theory as a piece of applied mathematics (cf. fluid dynamics) whose target phenomenon is the semantics of  ordinary mathematical discourse (cf. the behavior of real world substances like water).  More fully, model theory is to provide a workable account of the relations between the natural language of mathematics and a pre-theoretic metaphysics of structures that mathematics refer to (e.g.,  \cite{MR2993116}) or concepts that mathematicians express (e.g.,  \cite{MR1029277}).\footnote{See \cite[\S \S 6.2 and 6.4, respectively]{MR3821510}.  A remark on terminology:  applied mathematicians often describe themselves as providing `models' for worldly phenomena, but it would obviously invite confusion to describe model theory as providing a model of the mathematical portion of natural language.  We avoid this in the text with awkward circumlocutions like the one just employed -- model theory provides a `workable account' -- or it provides a `treatment' or some aspect of mathematical discourse is `idealized to' such and such. The reader is invited to interpret these passages as intending nothing other than `model' in sense of applied mathematics rather than that of model theory.}  In this way, the natural language of actual mathematics is idealized to a formal language $L$ and the structures mathematicians discuss to set-theoretic $L$-structures.\footnote{\label{b&w} Here, and in what follows, we elide the distinction between objects-modelism and concepts-modelism which Button and Walsh treat separately; reformulations of structure-talk into concept-talk should be relatively straightforward.  Also, the modelism in question is `moderate', that is, one that `rejects all appeals to faculties of \textit{mathematical intuition}, or anything similar' and appeals only to `human \textit{faculties} … which could plausibly both have evolved within a species, and also could have developed within an individual creature as it grew from a fetus into an adult' (\cite[p. 42]{MR3821510}).  Finally, these quotations all come in Button and Walsh's treatment of arithmetic rather than set theory, where the issue is the moderate modelist's appeal to the second-order version of Dedekind's categoricity theorem.  When it comes to set theory, where appeal to the second-order version of Zermelo's theorem is what's at issue, they write that `we would re-run all the arguments [from the arithmetic context], occasionally replacing the phrase ``arithmetic" with ``set theory".  Flogging this dead horse would, though, be exhausting and unilluminating.  So we shall pass over the corpse of moderate modelism without further ado' (ibid., pp. 181-182).}  Faced with the modelist's appeal to Zermelo's result in its second-order form (Theorem \ref{ZermeloCat}), Button and Walsh reach the familiar conclusion:

\begin{quote}
	invoking [the second-order theorem] is simply \textit{question-begging}, since the use of the full semantics simply \textit{assumes }precisely what was at issue … in appealing to full second-order logic, the … modelist is simply out of the frying pan, and into another frying pan.  \cite[pp. 159-160]{MR3821510}
\end{quote}

\noindent Internalism is then offered as an alternative perspective, from which internal categoricity theorems might gain better traction.

At the outset, all we're told about internalism is that it's a rejection of modelism and that it `relies heavily on deduction' \cite[p. 223]{MR3821510}.\footnote{The deduction they have in mind is a second-order, what we've been writing as $\vdash_2$.  Again (see footnote \ref{b&w}), the following attempt to characterize internalism is translated from the arithmetic context where it's explicit \cite[chapter 10]{MR3821510}, to the set-theoretic context where it's presumably implicit (ibid., chapter 11).}  Later, the contrast between modelism and internalism is glossed in terms of their respective formalizations of a pre-theoretic claim like `there is a set-theoretic structure':  the former offers `there is an $X$ and a binary relation $R$ on $X$ such that $(X, R) \models \ZFCTWO$', while the latter formulates a second-order formula conjoining the axioms of $\ZFCTWO$ as direct claims about variables $X$ and $R$ and opts for the second-order statement $\exists X \exists R \hspace{2pt}\ZFCTWO(X, R)$.  If we imagine each of these as the antecedent of a categoricity theorem, this happily coincides with our characterization of the external/weakly internal distinction as a property of theorems in Subsection~\ref{subsec:ded3} -- the former involves semantic notions and only mentions the axioms, while the latter involves no semantic notions and uses the axioms directly -- so there's a tight connection between Button and Walsh's internalism and weakly internal categoricity theorems.  

Beyond this, we're told that internalism is not if-thenism -- because the antecedent is asserted unconditionally -- not platonism -- because no stand is taken on the nature of the pre-theoretic metaphysics of structures -- and not logicism -- because it has no epistemological ambitions.\footnote{See \cite[pp. 236-237]{MR3821510}.}  And finally, `internalists need not fear the use of metalanguages ... [they] only oppose \textit{semantic} ascent' \cite[p. 238]{MR3821510} -- so they're perfectly open to proof-theoretic metatheory.   Still, for all this, internalism  remains somewhat underspecified, a point to which we return in Subsection~\ref{subsec:but3}.  

Given that Button and Walsh's project is exploratory, they can't be said to begin with a specific goal in mind, but they do end up with the conclusion that the mathematics they present  ... \footnote{\label{trans}For the record, Button and Walsh have a second, `transcendental argument' for a limited conclusion that a leading argument for the indeterminateness of CH `walks a dangerously thin path between falsity and incoherence' (\cite{MR3821510},  pp. 256).  Since this line of thought doesn't involve categoricity considerations, we leave it aside here.}

\begin{quote}
	... applies pressure to those who regard the continuum hypothesis as \textit{indeterminate}.  Indeed, it is not obvious how best to sustain that attitude.  (\cite{MR3821510}, p. 255)
\end{quote}

\noindent This stops short of a straightforward affirmation that CH has a determinate truth value, despite leaning in that direction.  We take this hedged conclusion as their goal, lay out the mathematics in Subsection~\ref{subsec:but2}, and evaluate its success in achieving that goal in  Subsection~\ref{subsec:but3}.  

\subsection{What do they actually do (or can be done)?}\label{subsec:but2}

We've seen that Button and Walsh take the semantic notions in Zermelo's Theorem \ref{ZermeloCat} to render it ineffective in a defense of the determinateness of CH, so naturally what they now pursue is an internal categoricity theorem.  We've also seen that their internalist is perfectly comfortable with syntactic second-order logic.\footnote{This is why they think Parsons should have settled for Theorem \ref{dedpure} (recall footnote \ref{BWonP}).}  Here they obviously differ from Parsons, a difference rooted in their respective diagnoses of what goes wrong in an attempt to apply Zermelo's theorem to CH in this way:  is the problem second-order logic itself or just the full semantics?  We won't attempt to adjudicate this dispute, but the result is that Parsons sees internalness as an instrument for achieving a first-order result, while Button and Walsh see internalness as an end in itself.  This is why Parsons needs a better version of Theorem \ref{theorem2} or Theorem \ref{theorem3}, while Button and Walsh would be content with an internal theorem for ZF$^2$ that runs parallel to  Theorem~\ref{folklore} for $\PATWO$ -- or rather, they want an internal theorem for $\ZFTWO$ running parallel to a version of  Theorem~\ref{folklore}  that's rephrased as a theorem of pure second-order logic.  

At the risk of pedantry, let's make this explicit.  First, suppose we've adapted the notational conventions laid out before Definition  \ref{definition1}  for the case of $\PATWO$ in the vocabulary $\{S,0\}$ and new symbols $N_1$, $N_2$, $S_1$, $S_2$, $0_1$, $0_2$.  Then, spelling out that definition, Theorem~\ref{folklore} becomes:

\begin{equation}\label{intermediate}
	\PA_1^{2(N_1)}\cup \PA_2^{2(N_2)}\vdash_2 \exists F\hspace{2pt}\Iso(F,N_1,N_2).
\end{equation}

\noindent As was noted in passing in Subsection~\ref{subsec:par2}, Definition \ref{definition1} takes the form it does to allow for infinite theories, so the formulation for the current case can be simplified by exploiting the finiteness of $\PATWO$.   Let $\PAF^{2(N_1)}_1=\bigwedge\{\phi^{(N_1)}(0_1,S_1):\phi(0,S)\in\PATWO\}\cup\{N_1(0_1),\forall x(N_1(x)\to N_1(S_1(x)))\}$ and $\PAF^{2(N_2)}_2=\bigwedge\{\phi^{(N_2)}(0_2,S_2):\phi(0,S)\in\PATWO\}\cup\{N_2(0_2),\forall x(N_2(x)\to N_2(S_2(x)))\}$. 
Then (\ref{intermediate}) becomes:\footnote{We treat the constant symbols $0,0_1,0_2$ as individual variables when they're quantified.}

\begin{theorem}[Theorem~\ref{folklore} restated, Internal categoricity for $\PATWO$]\label{folklore_restated}
$$\vdash_2\forall N_1,0_1,S_1,N_2,0_2,S_2((\PAF^{2(N_1)}_1\wedge\PAF^{2(N_2)}_2)\to \exists F\hspace{2pt}\Iso(F,N_1,N_2)).$$
\end{theorem} 

\noindent This is the arithmetic result that Button and Walsh hope to imitate for the case of set theory.\footnote{See \cite[p. 228]{MR3821510}.} 

Before leaving arithmetic, though, we should pause to ask precisely how Theorem \ref{folklore_restated} bears on the determinateness of arithmetic claims.  From a modelist's perspective, a categoricity theorem in model theory has implications for the pre-theoretic realm of mathematical structures -- namely, that the theory in question is satisfied by precisely one such structure\footnote{If these pre-theoretic structures can be isomorphic without being identical, then:  precisely one such structure up to isomorphism.} -- and it follows that each of its claims has a determinate truth value.\footnote{On the assumption in the previous footnote:  true in all such structures or false in all such structures – which is to say that each of its claims has a determinate truth value.}  But the internalist posits no such pre-theoretic metaphysics; the internal Theorem \ref{folklore_restated} involves no semantic notions.  As Button and Walsh put it, `\textit{no internal categoricity result can show that a theory pins down a unique $\mathcal{L}$-structure in the model-theorist's sense (even up-to-isomorphism)}' \cite[p. 229, emphasis in the original]{MR3821510}.  But an internal theorem can show what they call `intolerance'.

To see how that goes, consider this immediate consequence of Theorem~\ref{folklore_restated}:

\begin{corollary}\label{folklore_values}Suppose $\phi(0,S)$ is a second-order sentence in the vocabulary $\{0,S\}$. Then \\
$$
\begin{array}{l}
\vdash_2\forall N_1,0_1,S_1,N_2,0_2,S_2((\PAF^{2(N_1)}_1\wedge\PAF^{2(N_2)}_2)\to\\
\hspace{1cm}(\phi^{(N_1)}(0_1,S_1)\leftrightarrow\phi^{(N_2)}(0_2,S_2))).\end{array}$$
\end{corollary} 

\noindent Button and Walsh's `intolerance' follows:  

\begin{theorem}[\cite{MR3821510}]\label{intolerance}\emph{(`Intolerance' of $\PATWO$)} Suppose $\phi(0,S)$ is a second-order sentence in the vocabulary $\{0,S\}$. Then 
$$\vdash_2 \forall N,0,S(\PAF^{2(N)}\to\phi^{(N)}(0,S))\vee\forall N,0,S(\PAF^{2(N)}\to\neg\phi^{(N)}(0,S))).$$
\end{theorem}

\begin{proof} By Corollary~\ref{folklore_values}, 
$$\begin{array}{l}
\vdash_2\forall N_1,0_1,S_1,N_2,0_2,S_2((\PAF^{2(N_1)}_1\wedge\PAF^{2(N_2)}_2)\to\\
\hspace{1cm}(\phi^{(N_1)}(0_1,S_1)\vee\neg\phi^{(N_2)}(0_2,S_2))).
\end{array}$$ 
By 
rearranging  quantifiers and connectives, we obtain
{\renewcommand{\arraystretch}{1.2}
$$\begin{array}{l}
\vdash_2 \forall N_10_1S_1(\PAF^{2(N_1)}_1\to\phi^{(N_1)}(0_1,S_1))\vee\\
\hspace{0.4cm}\forall N_20_2S_2(\PAF_2^{2(N_2)}\to\neg\phi^{(N_2)}(0_2,S_2)))
\end{array}$$
\renewcommand{\arraystretch}{1}}
from which the claim follows by change of bound variables.
\end{proof}

\noindent Another way to put this is:

\begin{corollary} If $\phi$ is a second-order sentence in the vocabulary $\{0,S\}$, then  $$\PAF_1^{2(N_1)}\cup\PAF_2^{2(N_2)}\cup\{\phi^{(N_1)}(0_1,S_1),\neg\phi^{(N_2)}(0_2,S_2)\}$$ is deductively inconsistent.  
	
\end{corollary}

\noindent Button and Walsh take this to place `pressure on the Algebraic Attitude toward arithmetic'(\cite{MR3821510}, p. 236), that is, on the idea that arithmetic is not univocal.

This is the line of thought Button and Walsh want to see duplicated for the case of set theory.  Categoricity will have to be replaced by quasi-categoricity -- as in Theorem \ref{ZermeloCat}, the two models must be of the same height -- but with that proviso, just such an internal theorem has been proved by the second author and Tong Wang (\cite{MR3326591}).  Begin with $\ZFTWO$ in the usual vocabulary $\{\in\}$.\footnote{Button and Walsh actually use Scott-Potter set theory, but this makes little difference for our purposes, so we stick with $\ZFTWO$ to maintain consistency.   Urelements are similarly irrelevant, so for simplicity, we assume there are none.} Let $E_1$ and $E_2$ be binary relation symbols and $X_1$ and $X_2$ be unary relation symbols. If $\phi$ is a second-order sentence in the vocabulary $\{\in\}$, let $\phi(E_1)$ and $\phi(E_2)$ be translations of $\phi$ into the vocabularies $\{E_1\}$ and $\{E_2\}$, respectively, and  let $\phi^{(X_1)}(E_1)$ and  $\phi^{(X_2)}(E_2)$ be $\phi(E_1)$  and $\phi(E_2)$ be $\phi$ with its first- and second-order quantifiers relativized to $X_1$ and $X_2$, respectively. Once again exploiting finiteness, this time of $\ZFTWO$, let  $\ZFF^{2(X_1)}(E_1)=\bigwedge\{\phi^{(X_1)}(E_1):\phi(E_1)\in\ZFTWO(E_1)\}$ and 
$\ZFF^{2(X_2)}(E_2)=\bigwedge\{\phi^{(X_2)}(E_2):\phi(E_2)\in\ZFTWO(E_2)\}$.  Finally, let $\IA$ be the second-order sentence in the vocabulary $\{X_1,E_1,X_2,E_2\}$ which says -- assuming the axioms of $\ZFTWO$ for $(X_1,E_1)$ and $(X_2,E_2)$ -- that the classes of inaccessible cardinals in the sense of $(X_1,E_1)$ and $(X_2,E_2)$, respectively, are isomorphic, and let ISO$((X_1,E_1),(X_2,E_2)))$ be the second-order sentence that says  $(X_1,E_1)$ and $(X_2,E_2)$ are isomorphic. Then we have:

\begin{theorem}[\cite{MR3326591}]  \label{int_cat_ST}\emph{(Internal quasi-categoricity of ZF$^2$)}
	$$\vdash_2(\ZFF^{\hspace{1pt}2(X_1)}(E_1)\wedge\ZFF^{\hspace{1pt}2(X_2)}(E_2)\wedge \IA)\to\ISO((X_1,E_1),(X_2,E_2)).$$
\end{theorem} 

\noindent (Cf. \cite{MR3821510}, p. 255.) As in the case of arithmetic (see Subsection~\ref{subsec:par2}), this implies Zermelo's categoricity theorem.

The counterpart to Corollary \ref{folklore_values} follows immediately:

\begin{corollary}\label{folklore_values2}Suppose $\phi$ is a second-order sentence in the vocabulary $\{\in\}$. Then 
	$$\vdash_2(\ZFF^{\hspace{1pt}2(X_1)}(E_1)\wedge\ZFF^{\hspace{1pt}2(X_2)}(E_2)\wedge \IA)\to(\phi^{(X_1)}(E_1)\leftrightarrow\phi^{(X_2)}(E_2)).$$
\end{corollary} 

\noindent The counterpart to arithmetic intolerance -- the idea that every sentence is either true in all models or false in all models -- requires fixing the number of inaccessibles, their order type.  Any order type will do, so for the ultimate in simplicity, let $\IA_0$ be the first-order sentence of set theory saying that there aren't any inaccessible cardinals, that every limit cardinal $>\omega$ is singular. Finally, let $$\Gamma=\bigwedge(\ZF^{\hspace{1pt}2}\cup\{\IA_0\}).$$  Then, if $E$ is a new binary relation symbol and $X$ a new unary relation symbol, we have

\begin{theorem}[\cite{MR3821510}]\label{intolerance1} \emph{(`Intolerance' of ZF$^2$)} Suppose $\phi$ is a second-order sentence in the vocabulary $\{\in\}$. Then 
	$$\vdash_2 \forall X,E(\Gamma^{(X)}(E)\to\phi^{(X)}(E))\vee\forall X,E(\Gamma^{(X)}(E)\to\neg\phi^{(X)}(E)).$$
\end{theorem}

\begin{proof} By Corollary~\ref{folklore_values2}, 
	$$\vdash_2\forall X_1,E_1,X_2,E_2((\Gamma^{(X_1)}(E_1)\wedge\Gamma^{(X_2)}(E_2))\to(\phi^{(X_1)}(E_1)\vee\neg\phi^{(X_2)}(E_2))).$$ 
	By 
	rearranging  quantifiers and connectives, we obtain
	$$\vdash_2 \forall X_1,E_1(\Gamma^{(X_1)}(E_1)
	\to\phi^{(X_1)}(E_1))\vee\forall X_2,E_2(\Gamma^{(X_2)}(E_2)
	\to\neg\phi^{(X_2)}(E_2)))$$
	from which the claim follows by change of bound variables.
\end{proof}

\noindent Once again, `intolerance' can be rephrased:

\begin{corollary}\label{incon1}
	If $\phi$ is a second-order sentence in vocabulary  $\{\in\}$, then the theory 
	$$\{\Gamma^{(X_1)}(E_1),\Gamma^{(X_2)}(E_2),\phi^{(X_1)}(E_1),\neg \phi^{(X_2)}(E_2)\}$$
	is deductively inconsistent. 
\end{corollary}

\noindent From this, Button and Walsh draw the expected conclusion:  

\begin{quote}
	The present observation applies … pressure to those who regard the continuum hypothesis as \textit{indeterminate}.  Indeed, it is not obvious how best to sustain that attitude, in the face of this deductive inconsistency.  (\cite{MR3821510}, p. 255)
\end{quote}

\noindent We examine this claim in Subsection~\ref{subsec:but3}.

Theorem \ref{int_cat_ST} is both weakly internal and pure; of the desiderata present in Parsons's discussion, it fails only to be first-order.  Of course, this doesn't trouble Button and Walsh, but it's still worth observing that internal categoricity and its consequences can be carried over to first-order $\ZF$. The source of categoricity is actually more transparent in the first-order than in the second-order context because blanket appeal to the second-order comprehension axiom is replaced by explicit, targeted features of the first-order separation and replacement schemas. 

To see this, begin with first-order $\ZF$ in the vocabulary  $\{\in\}$ and use the notation $\ZF(E)$ for the result of replacing $\in$ with a binary relation symbol $E$.  Now suppose that $E_1$ and $E_2$ are new binary relation symbols, $X_1$ and $X_2$ are new unary relation symbols, and $\pi$ is a new unary function symbol. If $\phi(E)$ is a sentence in the vocabulary $\{E\}$, let $\phi^{(X)}(E)$ be $\phi(E)$ with the first-order quantifiers relativized to $X$.   Now let $\ZF^{(X_1)}(E_1)$ consist of all $\phi^{(X_1)}(E_1)$, where $\phi\in\ZF$, allowing in the separation and replacement schemas formulas from the vocabulary $\{X_1,E_1,X_2,E_2,\pi,\in\}$ with unrestricted (i.e. not relativized to $X_1$) quantifiers -- and similarly
$\ZF^{(X_2)}(E_2)$. Let $\IO_{\pi}$ be a first-order sentence of the vocabulary $\{X_1,E_1,X_2,E_2,\pi,\in\}$ saying -- assuming the axioms $\ZF$ for $(X_1,E_1)$ and $(X_2,E_2)$ -- that $\pi$ is an isomorphism between the ordinals of  $(X_1,E_1)$ and the ordinals of $(X_2,E_2)$. Finally, given a formula $\phi(x,y)$, let $\ISO_{\phi}((X_1,E_1),(X_2,E_2))$ be a first-order sentence that says -- again assuming the axioms $\ZF$ for $(X_1,E_1)$ and $(X_2,E_2)$ --  that $\phi(x,y)$ defines an isomorphism between $(X_1,E_2)$ and $(X_2,E_2)$, extending $\pi$.  Then

\begin{theorem}\label{int_cat_ST_first}\emph{(Internal quasi-categoricity of ZF)}
	There is a first-order formula $\phi=\phi(x,y)$ of set theory  such that $$\ZF \cup \ZF^{(X_1)}(E_1)\cup\ZF^{(X_2)}(E_2)\cup\{\IO_{\pi}\}\vdash\ISO_{\phi} ((X_1,E_1),(X_2,E_2)).$$
\end{theorem} 

\noindent The proof is as in \cite{MR3984974}.\footnote{\cite{Martin} can be read as giving an informal proof of something like this theorem.  A claim resembling Theorem \ref{int_cat_ST1} is formulated (without proof) in the appendix to  \cite{McGee1997-MCGHWL}.} Here the axioms for $(X_1,E_1)$ and $(X_2,E_2)$ are shifted to the left side of the turnstile because they are infinite in number, but this is trivial.  Theorem \ref{int_cat_ST_first} also has an extra ZF on the left side, like the extra $\PA$ in Theorem \ref{theorem2}, but this is less troublesome here, where we aren't worried about Kurt's epistemic limitations.   More to the current point, what fundamentally differentiates Theorem \ref{int_cat_ST_first} from the second-order Theorem \ref{int_cat_ST} is, as advertized, the mechanism by which  the crucial links between the $(X_1, E_1)$ and $(X_2,E_2)$ are forged:  in the first-order theorem, the key is allowing the vocabulary of one into the axiom schemas of the other; in the second-order theorem, these specifics are masked in the undifferentiated Comprehension Axioms.

As with arithmetic, the ZF on the left side can be eliminated if we assume that the domains of the two versions of $\ZF$ are the same and comprehensive.  An extra assumption like IO$_{\pi}$ is now unnecessary, but $\ISO_\phi$ needs modification:  given a formula $\phi(x,y)$, let $\Aut_{\phi}$ be the first-order sentence which says that $\phi(x,y)$ defines an automorphism between the binary predicates $E_1$ and $E_2$, again assuming the axioms $\ZF$ for $E_1$ and $E_2$. Let  $\ZF(E,E')$ be $\ZF$ with $E$ as the membership relation, but with separation and replacement schemas allowing formulas from the vocabulary $\{E,E'\}$. Then

\begin{theorem}[\cite{MR3984974}]\label{int_cat_ST1}
	There is a first-order formula $\phi=\phi(x,y)$ of set theory such that $$\ZF(E_1,E_2)\cup\ZF(E_2,E_1)\vdash
	\Aut_{\phi}(E_1,E_2).$$
\end{theorem} 

\noindent Here we have our counterpart to Theorem \ref{theorem3}.

As in the aftermath of Theorem \ref{int_cat_ST}, there's this immediate consequence:

\begin{corollary}\label{values}If $\psi$ is a first-order sentence in the vocabulary $\{\in\}$, then 
	$$\ZF(E_1,E_2)\cup\ZF(E_2,E_1)\vdash\psi(E_1)\leftrightarrow\psi(E_2).$$
\end{corollary} 

\noindent And a new version of Button and Walsh's intorance then follows.  Let $\ZF_n(E,E')$ be the conjunction of the first $n$ axioms of $\ZF(E,E')$ under some natural enumeration.  Then

\begin{theorem}[\cite{MR3821510}](\emph{`Intolerance' of ZF)}\label{intolerance2} There is a natural number $n$ such that if $\psi$ is a first-order sentence in the vocabulary $\{\in\}$, then 
	$$\vdash (\ZF_{n}(E_1,E_2)\to\psi(E_1))\vee(\ZF_{n}(E_2,E_1)\to\neg\psi(E_2)).$$
\end{theorem}

\begin{proof} By Theorem~\ref{int_cat_ST_first} there is an $n$ such that
	$$\ZF_n(E_1,E_2)\cup\ZF_n(E_2,E_1)\vdash\Aut_{\phi}(E_1,E_2).$$ 
Relying on $\Aut_{\phi}(E_1,E_2)$, induction on $\psi$ shows that 
	$$\vdash(\ZF_n(E_1,E_2)\wedge\ZF_n(E_2,E_1))\to(\psi(E_1)\leftrightarrow\psi(E_2)).$$
It follows by propositional logic that
	$$\vdash (\ZF_{n}(E_1,E_2)\to\psi(E_1))\vee (\ZF_{n}(E_2,E_1)
	\to\neg\psi(E_2)).$$
\end{proof}

\noindent And finally,

\begin{corollary}\label{incon2} If $\psi$ is a first-order sentence of the vocabulary $\{\in\}$, then the theory 
\begin{equation}\label{ded_incons1}\{\ZF(E_1,E_2),\ZF(E_2,E_1),\psi(E_1),\neg \psi(E_2)\}\end{equation}
is deductively inconsistent.
\end{corollary}

\subsection{Does what they did (or can be done) accomplish what they set out to do?}\label{subsec:but3}

To recap, Button and Walsh have argued that semantic second-order logic can't underpin a persuasive defense of the determinateness of CH, so the question is whether the syntactic second-order logic of the  internal categoricity theorems 
 of Subsection~\ref{subsec:but2} can do better.  They note that `no internal categoricity result can show that a theory pins down a unique L-structure in the model-theorist's sense (even up-to-isomorphism)' (\cite{MR3821510}, p. 229), and for that reason, these results provide no direct comfort to the modelist. Still, they argue, taking the internalist's perspective, Theorems \ref{int_cat_ST} and \ref{intolerance1} and their corollaries make concern over the determinateness of CH difficult to sustain.\footnote{Recall the quotation from \cite{MR3821510}, p. 255, displayed at the end of Subsection~\ref{subsec:but1}.} 
    Here it's worth noting that Theorem \ref{intolerance1} only says that a certain disjunction is provable in syntactic second-order logic, which of course doesn't mean that either disjunct need be provable.  From this theorem, the internalist can infer that any model of the second-order axioms must satisfy the disjunction.  Continuing the internalist's line of thought, this can only mean that any Henkin model satisfies it -- which obviously leaves open the possibility that some such models satisfy one and some the other disjunct.  This would certainly appear sufficient to sustain reasonable concern about the determinateness of CH.  (We suggest below that this isn't the true source of internalistic                   worries about CH.)

    In any case, Button and Walsh's actual discussion takes a different approach to their goal, attempting not to establish
%
%
%
the determinateness of CH but to defuse the inclination to think otherwise:

\begin{quote}
	It is doubtful that we can understand the claim that there are multiple `equally preferable' models, some of which satisfy the continuum hypothesis and others which do not, and hence the supposed \textit{motivation} for thinking that the continuum hypothesis is indeterminate.  \cite[p. 256]{MR3821510}
\end{quote}

\noindent What they have in mind here is a concern about CH arising from the fact that ZF$^2$ admits non-isomorphic Henkin models, and they claim that this argument has been undercut.\footnote{See \cite[pp. 255-256]{MR3821510}. It seems this undercutting is largely based on the `transcendental' considerations of their chapter 9, rather than Intolerance and Deductive Inconsistency.  (Recall footnote \ref{trans}.)}  Of course, this is a modelist motivation:  the model-theoretic fact of non-isomorphic models is taken to reflect the existence of non-isomorphic pre-theoretic metaphysical structures.  But a modelist case for the determinateness of CH has already been rejected, and this argument from non-isomorphic models to metaphysical structures is not one any self-respecting internalist would make.  If an internalist were concerned about the status of CH, it would be for other reasons, perhaps reasons not assuaged by categoricity, intolerance, or deductive inconsistency.  

This points to a more fundamental problem:  Button and Walsh haven't given internalism a fair hearing, or so we claim.  If modelism is an understanding of model theory as an applied mathematical treatment of the semantics of ordinary mathematical discourse, then internalism should be an understanding of proof theory as an applied mathematical treatment of the proving activity of mathematicians – perhaps a better counterpart to the term `modelism'  would be   `proofism'.\footnote{\label{parallel1}This parallel has its limits, two of which we note for the record, one here and another in footnote \ref{parallel2}. First is the observation that the applied math story is more plausible in the case of proof theory than  of model theory.  A modelist, noting that first-order $\PA$ has non-standard models, draws as an immediate consequence that the pre-theoretic metaphysics includes corresponding non-standard structures, but applied mathematics doesn't work that way:  the continuity of an ideal fluid is known not to carry over to actual water; applied mathematicians take great care to determine which aspects of their treatments correspond to the worldly situation and which don't, to determine which idealizations are safe in which contexts.  Likewise for the proofist:  the existence of a proof-theoretic `proof' isn't taken as conclusive evidence of the possibility, let alone the actuality, of a real world proof.  This suggests that proof theory is answerable to a worldly phenomenon capable of pushing back, while model theory is not.}  These two represent the poles of `mathematics as descriptive activity' vs. `mathematics as proving activity', depending on which is taken as fundamental.  The modelist will regard proving activity as instrumentally important, as epistemology, a way of gaining insight into the pre-theoretic structures.  The proofist will regard ordinary mathematical language as meaningful, but Austin and Wittgenstein have taught us that a discourse can be meaningful without being referential. The proofist's practice is purely syntactic in the sense that it involved no appeal to a subject matter (and in particular, no appeal to the modelist's structures).\footnote{\label{parallel2}The second limitation of the parallel between the modelist and the proofist concerns their views of proof theory and model theory, respectively.  Both regard proof theory as an applied mathematical account of real world proving activity.  Where they differ is on model theory, which the modelist also sees as an applied mathematical account, this time of the semantics of real world mathematical discourse.  The proofist doesn't think real world mathematical discourse has a semantics in the modelist's sense -- it's not referential, it isn't out to describe or discover truths about some robust subject matter.  It's simply a practice of devising and deploying mathematically rich concepts and theories, some in reaction to (presumably) shared intuitive pictures, some for purely mathematical ends, some for the purposes of natural science.  Likely due to the structure of human cognition (see, e.g., \cite{MR3930026}), these concepts and theories are expressed in terms of things with properties, things in relations, and the like.  Model theory takes place within this practice, within the imagined ontology of the practice, as an account of the relations between a formalized version of ordinary language of the discourse and the things the theory itself has imagined.  So a formal language, $L$, corresponds to ordinary mathematical discourse, just as the modelist has it, but the corresponding $L$-structure, made of elements of the imagined ontology, corresponds to nothing real.  This ploy has foundational uses -- in independence proofs, for just one example -- and its own particular interest as a branch of mathematics in its own right.}  

Much could be said about this proofism,\footnote{Some of it is said in footnotes \ref{parallel1} and \ref{parallel2}.  For allied views, see the `Enchanced If-thenism' of \cite{maddy:toappear} or the `Arealism' of \cite{MR2779203}.} but for our purposes, the central point is that `does CH have a determinate truth value?' is an essentially semantic question, not the sort of thing the proofist would (or could) ask.  If we fully commit to the proofist perspective, the salient question about CH isn't whether it has a determinate truth value -- the modelist's truth and existence simply aren't part of this discourse.  The question that arises for the proofist is  whether the set-theoretic community will ever hit upon an axiom (or other unforseen type of development) with sufficient mathematical advantages to merit adoption -- and which implies CH or implies $\neg$CH. (The worry that this may never happen would be the source of a proofist's concern about the status of CH.) Perhaps not coincidentally, this is what many set theorists who engage with foundations are actively trying to do,  whatever their metaphysical or semantics beliefs might be.  The ultimate fate of this endeavor remains uncertain, but one point is clear:  internal categoricity theorems provide no assurance of success.

We conclude that Button and Walsh have not succeeded in establishing that internalist (proofist) concerns over the status of CH are `difficult to sustain'.  

\section{Conclusion}\label{sec:con}

We've now surveyed five different attempts to put categoricity results to philosophical use.  The efforts of Dedekind, Zermelo, and Kreisel we take to have been successful; those of Parsons and Button-Walsh less so.  Can any general moral be drawn from these observations?  Is there a distinguishing feature that accounts for the disparity between the first group and the second?

One obvious difference is that Parsons and Button-Walsh focus on internal theorems and carefully police the background theory in which they're proved, for the sake of a kind of purity:  first-order $\PA$ (Parsons) or syntactic second-order logic (Button-Walsh).  These features or their absence simply aren't salient for the earlier authors, so it's worth asking what explains their centrality for Parsons and Button-Walsh.  The answer seems to lie in the fact that their respective goals are qualitatively different from those of their predecessors.  Parsons is out to show the determinateness of our concept of natural number, the uniqueness of the structure it purportedly describes (up to isomorphism).  Button and Walsh hope to establish that it's at least extremely difficult to maintain that CH lacks a determinate truth value.  Both these projects concern an ambient pre-theoretic metaphysics that stands in relation to ordinary mathematical discourse (referred to, expressed by).  In contrast, the efforts of Dedekind, Zermelo, and Kreisel are closely tied to mathematical ambitions: Dedekind demonstrates that the concept of natural number can be characterized without appeal to spatio-temporal intuition; Zermelo shows how the set-theoretic axioms, especially Foundation, generate a streamlined and fruitful characterization of the universe of sets; Kreisel points out that the newly established independence of CH is of a new variety, qualitatively different from that of previous examples like large cardinal axioms.  These are all significant philosophical advances, but none concern the semantics of ordinary mathematical discourse or its purported, pre-theoretic subject matter.  

One last question: if these new pure internal categoricity theorems don't establish what recent writers had hoped, are they without foundational significance?  Perhaps unsurprisingly, we think the first-order theorems do make an important philosophical point: an outcome that was thought to require second-order resources – namely, categoricity theorems – can actually be achieved by suitable first-order means.  What seems to be crucial is a link between the languages of the two relevant models -- whether hidden in the second-order comprehension axioms of Theorems \ref{folklore_restated} and \ref{int_cat_ST} or explicit in the assumptions of Theorems \ref{theorem2} and \ref{int_cat_ST_first}. This is a useful discovery, which supports our general moral:  a bit of mathematics that fails at one task might succeed (and even be aimed) at another.

\appendix  

\nocite{Benacerraf1983-BENPOM-18,MR1851828}

\bibliography{intcat}

\subsection*{Acknowledgements}

The first author would like to thank Charles Leitz, Chris Mitsch, Stella Moon, Jeffrey Schatz, and Evan Sommers for helpful comments on an earlier draft.  The second  author would like to thank  the Academy of Finland, grant no: 322795. This project has received funding from the European Research Council (ERC) under the European Union’s Horizon 2020 research and innovation programme (grant agreement No 101020762).
\end{document}